\documentclass{conm-p-l}

\usepackage{biblatex}
\usepackage{etoolbox}

\usepackage{tikz}
\usetikzlibrary{calc}
\usetikzlibrary{cd}
\usetikzlibrary{decorations.markings}
\usetikzlibrary{decorations.text} 
\usetikzlibrary{decorations.pathreplacing}
\usetikzlibrary{decorations.pathmorphing}
\usetikzlibrary{arrows.meta}

\usepackage{longtable}
\usepackage{subcaption}
\usepackage{amssymb}
\usepackage{amsmath}
\usepackage{bbm}
\usepackage{yhmath} 
\usepackage{uniinput}
\usepackage{amsthm}
\usepackage{stmaryrd}
\usepackage{mathtools}
\usepackage{fouridx}
\usepackage{xparse}
\usepackage{bm}
\usepackage{aliascnt}
\usepackage{hyperref}
\usepackage{sansmath}
\usepackage[bbgreekl]{mathbbol}

\newcommand{\<}{\langle}
\renewcommand{\>}{\rangle}
\newcommand{\C}{\mathbb{C}}
\newcommand{\N}{\mathbb{N}}
\newcommand{\cT}{\mathcal{T}}

\newcommand{\Id}{\operatorname{Id}}
\newcommand{\id}{\mathbbm{1}}

\newcommand{\Hom}{\operatorname{Hom}}
\newcommand{\Ker}{\operatorname{Ker}}

\let\emph\textbf

\newcommand{\htensor}{\widehat{\otimes}}

\newcommand{\MC}{\operatorname{\mathrm{MC}}}

\newcommand{\Gtl}{\operatorname{Gtl}}

\def\H{\operatorname{H}}

\newcommand{\HH}{\operatorname{HH}}
\newcommand{\HC}{\operatorname{CC}}

\newcommand{\odd}{\operatorname{odd}}
\newcommand{\even}{\operatorname{even}}

\def\Im{\operatorname{Im}}
\newcommand{\matf}[4]{MATF} 

\renewcommand{\k}{\mathbbm{k}}

\theoremstyle{definition}
\newtheorem{theorem}{Theorem}

\newaliascnt{remark}{theorem}
\newaliascnt{definition}{theorem}
\newaliascnt{proposition}{theorem}
\newaliascnt{lemma}{theorem}
\newaliascnt{corollary}{theorem}
\newaliascnt{example}{theorem}
\newaliascnt{convention}{theorem}
\newaliascnt{todo}{theorem}

\newtheorem{remark}[remark]{Remark}
\newtheorem{definition}[definition]{Definition}
\newtheorem{proposition}[proposition]{Proposition}
\newtheorem{lemma}[lemma]{Lemma}
\newtheorem{corollary}[corollary]{Corollary}
\newtheorem{example}[example]{Example}

\aliascntresetthe{remark}
\aliascntresetthe{definition}
\aliascntresetthe{proposition}
\aliascntresetthe{lemma}
\aliascntresetthe{corollary}
\aliascntresetthe{example}
\aliascntresetthe{convention}
\aliascntresetthe{todo}

\numberwithin{equation}{section}
\numberwithin{figure}{section}
\numberwithin{table}{section}

\makeatletter\let\c@table\c@figure\makeatother

\numberwithin{theorem}{section}
\numberwithin{remark}{section}
\numberwithin{definition}{section}
\numberwithin{proposition}{section}
\numberwithin{lemma}{section}
\numberwithin{corollary}{section}
\numberwithin{example}{section}
\numberwithin{convention}{section}
\numberwithin{todo}{section}

\begin{document}

\title{Deformations of Gentle $A_\infty$-Algebras}

\author{Raf Bocklandt}
\address{KdV Institute, University of Amsterdam, Amsterdam, The Netherlands}
\email{r.r.j.bocklandt@uva.nl}
\thanks{The authors were supported in part by the NWO-grant 'Grant Algebraic methods and structures in the theory of Frobenius manifolds and their
applications' (TOPl.17.012).}

\author{Jasper van de Kreeke}
\address{KdV Institute, University of Amsterdam, Amsterdam, The Netherlands}
\email{j.d.c.vandekreeke3@uva.nl}

\date{\today}



\begin{abstract}
In this paper we calculate the Hochschild cohomology of gentle $A_\infty$-algebras of arc collections on marked surfaces. When the underlying 
arc collection has no loops or two-cycles, we show that the dgla structure of the Hochschild complex is formal and give an explicit realization of all deformations up to gauge equivalence.
\end{abstract}

\maketitle
\tableofcontents
\section{Introduction}

To a collection of oriented arcs on a closed surface with marked points, one can associate an $A_\infty$-algebra called the gentle $A_\infty$-algebra. These were introduced in \cite{Bocklandt,HKK} to study wrapped Fukaya categories of punctured surfaces \cite{abouzaid2013homological}, where the punctured surface is obtained by removing the marked points.
These algebras are also variants of well-known algebras that are studied in representation theory \cite{Assem,opper2018geometric,lekili2020derived}.

In this paper we study the deformation theory of these $A_\infty$-algebras in detail. The intuition from mirror symmetry tells us that deforming these algebras as curved $A_\infty$-algebras should correspond to filling in these punctures with normal points or orbifold points \cite{seidel2008homological,efimov2009homological}. 

We will work out this idea as follows. Starting from an arc collection on a closed surface with marked points, we introduce the combinatorial notion of an orbigon and use it to define a family of higher products that count these orbigons. We show that these structures satisfy the curved $A_\infty$-axioms and indeed deform the gentle $A_\infty$-algebra coming from the arc collection.

To show that these deformations solve the deformation problem we look at the Hochschild cohomology. We show that each element in the Hochschild cohomology is determined by its zeroth and first component; in other words its nullary and unary product.  In the case that the arc collection has no loops or two-cycles we obtain an explicit basis for the Hochschild cohomology. This enables us to give a description of the Gerstenhaber algebra structure on the Hochschild cohomology.

These computations match work done by Wong \cite{wong2021dimer} on the Borel-Moore  cohomology for matrix factorizations of dimer models (which is the B-side analogon of our setting, from the perspective of mirror symmetry \cite{Bocklandt}), as well as the theme that Hochschild cohomology of wrapped Fukaya categories agrees with symplectic homology \cite{Ganatra, Ritter-Smith}.

Finally we show that the bracket from the Gerstenhaber structure is formal and conclude that each solution of the Maurer-Cartan equation for the Hochschild cohomology is gauge equivalent to one of the curved $A_\infty$-structures we introduced.

\section{Curved gentle algebras}

\newcommand{\cA}{\mathcal{A}}
\newcommand{\Z}{ℤ}
\newcommand{\Int}{\operatorname{Int}}
\newcommand{\old}[1]{[OLD] #1}
\newcommand{\rmu}{{}^r\bbmu}
\newcommand{\ord}{\mathrm{ord}}

In this section we will introduce curved gentle algebras. These are curved deformations of the gentle $A_\infty$-algebras that were defined in \cite{Bocklandt} and \cite{HKK}. 

\subsection{Arc collections and gentle algebras.}
\begin{definition}
A \emph{marked surface} $(S,M)$ is a pair consisting of a compact oriented surface without boundary $S$ and a finite subset of marked points $M\subset S$. We will denote the genus of the surface by $g$. Furthermore we will assume that $ |M|≥3 $ if $ S $ is a sphere and $ |M|≥1 $ otherwise.

An \emph{arc collection} $\cA$ is a set of oriented curves $a: [0,1] \to S$ that meet $M$ only at the end points ($a^{-1}(M)=\{0,1\}$) and do not (self)-intersect internally.  We say that an arc collection \emph{splits the surface} if the complement of the arcs is the disjoint union of discs. These disks are called the \emph{faces} and we denote the set of faces by $F$. Each face can be seen as a polygon bounded by arcs.

An arc collection that splits the surface satisfies
\begin{itemize}
\item
the \emph{no monogons or digons condition [NMD]}, if no single arc or pair of arcs bounds a face. In this case the surface is split in $n$-gons with $n\ge 3$.
\item the \emph{no loops or two cycles condition [NL2]} if every arc has two different end points and no two arcs share more than one end point. 
\item
the \emph{dimer condition}, if the arcs around each disk form an oriented cycle. 
In case of a dimer, we will call a face positive or negative depending on whether the cycle around it is anticlockwise or clockwise. We can partition $F$ accordingly: $F=F^+\cup F^-$.
\end{itemize}

From now on we will assume that [NMD] holds, but in later sections we will sometimes have to assume the stronger [NL2] condition.
\end{definition}

Recall that a \emph{quiver} is a four-tuple $Q=(Q_0,Q_1,h,t)$ representing an oriented graph with vertices $Q_0$, arrows $Q_1$ and maps $h,t:Q_1\to Q_0$ that assign to each arrow its head and tail. A quiver is graded by a group $G$ if it comes with a map $|\cdot|: Q_1 \to G$.

The \emph{path algebra} $\C Q$ of a quiver $Q$ is the complex vector space spanned by the paths with as product concatenation of paths if possible and zero otherwise. We write the arrows as going from right to left so $\alpha\beta$ is a genuine path if $t(\alpha)=h(\beta)$. Every vertex of the quiver $v\in Q_0$ gives rise to a path of length zero, which is called the vertex idempotent $\id_v$. These span a semisimple subalgebra which we will denote by $\k\cong \C^{Q_0}$. If $Q$ is $G$-graded then $\C Q$ is a graded $\k$-algebra by assigning to each arrow its $G$-degree and to each vertex idempotent degree $0$. 

\vspace{.3cm}

It is tempting to consider the arc collection itself as a quiver, but we will not do this. Instead we will consider a different quiver, which is obtained by putting a vertex in the middle of each arc.
\begin{definition}
Given an arc collection $\cA$ that splits $(S,M)$, we define a $\Z_2$-graded quiver $Q_\cA$ as follows.
\begin{itemize}
    \item The vertices of the quiver are the arcs: $(Q_\cA)_0=\cA$.
    \item For each angle of a face we define an arrow that corresponds to the internal anticlockwise angle between consecutive arcs that meet in that corner. 
    \item An arrow has degree zero if both arcs have the same direction at the marked point (both outwards or both inwards), and degree one otherwise. Note that $\cA$ is a dimer if and only if all arrows have $\Z_2$-degree $1$.
\end{itemize}
\end{definition}
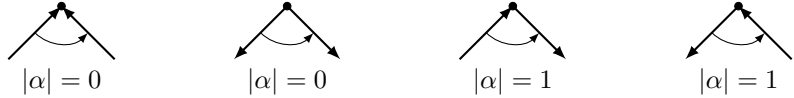
\begin{center}
\begin{tikzpicture}
\begin{scope}
  \draw[fill=black] (0,0) circle (.05cm);
\draw[thick,-latex] (225:1)--(0,0);
\draw[thick,-latex] (315:1)--(0,0);
\draw[-latex] (225:.5) arc (225:315:.5);
\draw (0,-1) node{$|\alpha|=0$};
\end{scope}

\begin{scope}[xshift=3cm]
  \draw[fill=black] (0,0) circle (.05cm);
\draw[thick,latex-] (225:1)--(0,0);
\draw[thick,latex-] (315:1)--(0,0);
\draw[-latex] (225:.5) arc (225:315:.5);
\draw (0,-1) node{$|\alpha|=0$};
\end{scope}

\begin{scope}[xshift=6cm]
  \draw[fill=black] (0,0) circle (.05cm);
\draw[thick,-latex] (225:1)--(0,0);
\draw[thick,latex-] (315:1)--(0,0);
\draw[-latex] (225:.5) arc (225:315:.5);
\draw (0,-1) node{$|\alpha|=1$};
\end{scope}

\begin{scope}[xshift=9cm]
  \draw[fill=black] (0,0) circle (.05cm);
\draw[thick,latex-] (225:1)--(0,0);
\draw[thick,-latex] (315:1)--(0,0);
\draw[-latex] (225:.5) arc (225:315:.5);
\draw (0,-1) node{$|\alpha|=1$};
\end{scope}
\end{tikzpicture}
\end{center}

We will denote these `angle arrows' by Greek letters and use $h(\alpha)$ and $t(\alpha)$ to denote the arcs (vertices) that correspond to the head and tail of the arrows. Each angle arrow also turns around a unique marked point and is contained in a unique face. We will denote these by $m(\alpha)$ and $f(\alpha)$.
In this quiver two consecutive arrows either are angles that turn around a common marked point or they correspond to consecutive angles in a face.
\[
t(\alpha)=h(β) \implies m(\alpha)=m(\beta) \text{ or }f(\alpha)=f(\beta)
\]

\begin{definition}
The \emph{gentle algebra} $\Gtl_\cA$ of an arc collection $\cA$ is the path algebra of $Q_{\cA}$ modulo the ideal of relations spanned by the products of arrows that are consecutive angles in a face. 
\[
\Gtl_\cA = \C Q_{\cA}/\<\alpha\beta | α \text{ and } β \text{ are consecutive angles in a face}\>
\]
The product rule in this algebra can be illustrated pictorially as follows.
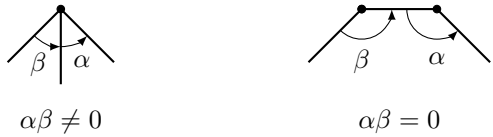
\begin{center}
 \begin{tikzpicture}
 \begin{scope}
  \draw[fill=black] (0,0) circle (.05cm);
  \draw[thick] (225:1)--(0,0);
  \draw[thick] (270:1)--(0,0);
  \draw[thick] (315:1)--(0,0);
  \draw[-latex] (225:.5) arc(225:270:.5);
  \draw[-latex] (270:.5) arc(270:315:.5);
  \draw (292:.75) node {$\alpha$};
  \draw (248:.8) node {$\beta$};
  
\draw (0,-1.5) node{$\alpha\beta\ne 0$};
   \end{scope}

 \begin{scope}[xshift=4cm]
  \draw[fill=black] (0,0) circle (.05cm);
  \draw[fill=black] (1,0) circle (.05cm);
  \draw[thick] (225:1)--(0,0);
  \draw[thick] (0,0)--(1,0);
  \draw[thick,xshift=1cm] (0,0)--(315:1);
  \draw[fill=black] (0,0) circle (.05cm);
  \draw[-latex] (225:.4) arc (225:360:.4);
  \draw[-latex,xshift=1cm] (180:.4) arc (180:315:.4);
   \draw (1,-.65) node {$\alpha$};
  \draw (0,-.7) node {$\beta$};
 
\draw (.5,-1.5) node{$\alpha\beta=0$};
   \end{scope}
   \end{tikzpicture}
\end{center}
\end{definition}

\begin{remark}
The number of arrows is $2|\cA|$ because in each vertex precisely two arrows arrive and two arrows leave.  More precisely for every angle arrow $\alpha$ there are precisely two angle arrows $\beta_1$ and $\beta_2$ with $h(\alpha)=t(\beta_i)$. One will satisfy $\beta_1\alpha=0$ and for the other one we have $\beta_2\alpha\ne 0$. Likewise there are two $\gamma_i$ with $h(\gamma_i)=t(\alpha)$, one with $\alpha \gamma_1=0$ and one with $\alpha \gamma_2\ne 0$. This notion of a gentle algebra refers to this property and derives from representation theory \cite{Assem}. The usual definition of a gentle algebra also entails that the algebra is finite-dimensional but this is not the case for our algebras. It is possible to obtain finite dimesnional algebras by looking at surfaces with marked points on the boundary \cite{lekili2020derived,opper2018geometric}.
\end{remark}

\begin{example}\label{conifex}
In the picture below on the left you see an arc collection on a torus with two marked points. It has $4$ arcs and two square faces. The quiver $Q_\cA$ has $4$ vertices and $8$ angle arrows. The angle arrows all have degree $1$. The relations are generated by all paths $\alpha_i\beta_j$ and $\beta_i\alpha_j$, whenever the arrows are composable. This implies that every nonzero path in the gentle algebra is either a sequence of only $\alpha's$ or only $\beta's$. The former turn around the first marked point, while the latter turn around the second marked point.

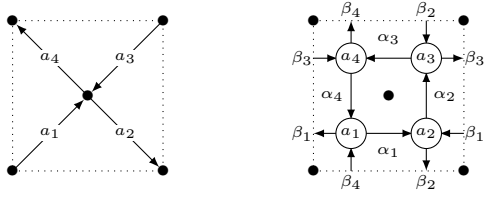
\begin{center}
\begin{tikzpicture}
\draw[dotted] (0,0)--(2,0)--(2,2)--(0,2)--cycle;
\draw [-latex,shorten >=2pt] (0,0) to node [rectangle,fill=white,inner sep=1pt] {{\tiny $a_1$}} (1,1);
\draw [-latex,shorten >=2pt] (2,2) to node [rectangle,fill=white,inner sep=1pt] {{\tiny $a_3$}} (1,1);
\draw [-latex,shorten >=2pt] (1,1) to node [rectangle,fill=white,inner sep=1pt] {{\tiny $a_2$}} (2,0);
\draw [-latex,shorten >=2pt] (1,1) to node [rectangle,fill=white,inner sep=1pt] {{\tiny $a_4$}} (0,2);
\node at (0,0) {$\bullet$}; 
\node at (2,0) {$\bullet$}; 
\node at (0,2) {$\bullet$}; 
\node at (2,2) {$\bullet$}; 
\node at (1,1) {$\bullet$};

\begin{scope}[xshift=4cm]
\draw[dotted] (0,0)--(2,0)--(2,2)--(0,2)--cycle;
\draw [-latex] (.5,.5)--(0,.5);
\draw [-latex,shorten >=5pt] (2,.5)--(1.5,.5);
\draw [-latex] (1.5,.5)--(1.5,.0);
\draw [-latex,shorten >=5pt] (1.5,2)--(1.5,1.5);
\draw [-latex] (1.5,1.5)--(2,1.5);
\draw [-latex,shorten >=5pt] (0,1.5)--(.5,1.5);
\draw [-latex] (.5,1.5)--(.5,2);
\draw [-latex,shorten >=5pt] (.5,0)--(.5,.5);

\draw [-latex,shorten >=5pt] (.5,.5) --(1.5,.5);
\draw [-latex,shorten >=5pt] (1.5,.5) --(1.5,1.5);
\draw [-latex,shorten >=5pt] (1.5,1.5) --(.5,1.5);
\draw [-latex,shorten >=5pt] (.5,1.5) --(.5,.5);

\draw (1,.25) node{{\tiny $\alpha_1$}};
\draw (1,1.75) node{{\tiny $\alpha_3$}};
\draw (1.75,1) node{{\tiny $\alpha_2$}};
\draw (.25,1) node{{\tiny $\alpha_4$}};

\draw (-.15,.5) node{{\tiny $\beta_1$}};
\draw (2.15,.5) node{{\tiny $\beta_1$}};
\draw (-.15,1.5) node{{\tiny $\beta_3$}};
\draw (2.15,1.5) node{{\tiny $\beta_3$}};
\draw (.5,-.15) node{{\tiny $\beta_4$}};
\draw (.5,2.15) node{{\tiny $\beta_4$}};
\draw (1.5,-.15) node{{\tiny $\beta_2$}};
\draw (1.5,2.15) node{{\tiny $\beta_2$}};

\node at (.5,.5) [circle,draw,fill=white,minimum size=10pt,inner sep=1pt] {\mbox{\tiny $a_1$}}; 
\node at (1.5,.5) [circle,draw,fill=white,minimum size=10pt,inner sep=1pt] {\mbox{\tiny $a_2$}}; 
\node at (1.5,1.5) [circle,draw,fill=white,minimum size=10pt,inner sep=1pt] {\mbox{\tiny $a_3$}}; 
\node at (.5,1.5) [circle,draw,fill=white,minimum size=10pt,inner sep=1pt] {\mbox{\tiny $a_4$}}; 

\node at (1,1) {$\bullet$};
\node at (2,2) {$\bullet$};
\node at (0,0) {$\bullet$};
\node at (2,0) {$\bullet$};
\node at (0,2) {$\bullet$};

\end{scope}
\end{tikzpicture}
    
\end{center}
\end{example}
In general, the nonzero paths of a gentle algebra correspond to positive angles between arcs that share a marked point. Every nonzero angle paths is uniquely determined by it sequence of angle arrows and products of angles paths turning around different marked points are zero. 

For each $m \in M$ we define 
    \[
       \ell_m = \alpha_1\dots\alpha_k + \dots + \alpha_k\dots\alpha_1
    \]
where $\alpha_1,\dots,\alpha_k$ are the angle arrows around $m$ ordered in anticlockwise direction.
These elements represent single loops around the marked points and they generate the center of the algebra. 
\begin{lemma}\label{center}
\[
Z(\Gtl_\cA) = ℂ[\ell_m| m \in M]/(\ell_i\ell_j| i\ne j).
\]
\end{lemma}
\begin{proof}
One can easily check that the $\alpha\ell_m=\ell_m\alpha$ if $m(\alpha)=m$ and $\alpha\ell_m=\ell_m\alpha=0$ if $m(\alpha)\ne m$. This also implies that $\ell_u\ell_v=0$ if $u\ne v$.

Suppose $z=c\beta +\dots$ is a nonzero central element containing the angle path $\beta$ and let $\alpha$ be the angle arrow that follows $\beta$ in the cycle around $m$. Then we have that $\alpha \beta\ne 0$, so if $\alpha z=z\alpha$ then $\alpha\beta$ must end in $\alpha$ and hence $\beta$ must be a cycle that winds a number of full turns around $m$. Therefore $z$ will contain $\ell_m^r$ and the $\ell_m^r$ form a basis for the part of the center with path length $>0$. The length $0$ part of the center is $ℂ$ because the quiver is connected. 
\end{proof}

\begin{lemma}\label{derivations}
As a $Z(\Gtl_\cA)$-module the outer ($\Z_2$-graded) $\k$-derivations are generated by the Euler derivations
\[
E_\alpha := \alpha\partial_\alpha.
\]
\end{lemma}
\begin{proof}
Let $\alpha$ be an angle arrow.
If $d$ is a $\k$-derivation then $h(d\alpha)=h(\alpha)$ and $t(d\alpha)=t(\alpha)$. If $d\alpha$ contains a term $\gamma$ that does not turn around the same marked point as $\alpha$, let $\beta$ be the angle such that $\gamma\beta\ne 0$ is cyclic. We then have that $\alpha\beta=0$ because they turn around different marked points. Therefore
\[
0 = d(\alpha\beta) = \gamma\beta+ \dots  \pm \alpha d\beta,
\]
but this is impossible because these two components turn around different marked points and hence cannot cancel each other.

So suppose $d\alpha$ turns around the same marked point $m$ as $\alpha$.
In that case there are the following possibilities. 
\begin{enumerate}
    \item If $a= h(\alpha)=t(\alpha)$ then either $a$ forms a monogon (which is excluded by condition [NMD]) or $\alpha=\ell_m$ is a full turn around a marked point $m$ with one arc $a$ arriving. 
    
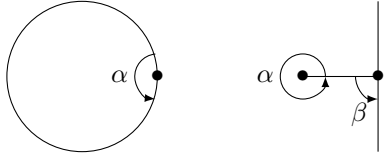
\[
    \begin{tikzpicture}
    \draw (0,0) node {$\bullet$};
    \draw (-1,0) circle (1cm);
    \draw[-latex] (95:.3) arc (95:265:.3);
    \draw (-.5,0) node{$\alpha$};
    \end{tikzpicture}\hspace{1cm}
    \begin{tikzpicture}
    \draw (0,0) node {$\bullet$};
    \draw (1,0) node {$\bullet$};
    \draw (0,0)--(1,0);
    \draw (1,1)--(1,-1);
    \draw[-latex] (0:.3) arc (0:360:.3);
    \draw (-.5,0) node{$\alpha$};
    \draw[-latex, xshift=1cm] (180:.3) arc (180:270:.3);
    \draw (.75,-.5) node{$\beta$};
   \end{tikzpicture}
    \]
    In the latter case $d\alpha = g(\alpha)\id_a$ for some polynomial $g$. This polynomial cannot have a constant term: let $\beta$ be the arrow that follows $\alpha$ in the face $f(\alpha)$. Then $\alpha$ and $\beta$ turn around different marked points so
    \[
    0=d(\beta\alpha)= (d\beta)\alpha \pm \beta g(\alpha)  = \pm g_0\beta.
    \]
    So $d\alpha = (g(\alpha)\alpha^{-1})\alpha= zE_\alpha(\alpha)$ for $z=g(\ell_m)\ell_m^{-1}$.
    \item If $a= h(\alpha)\ne t(\alpha)$ and $a$ is a loop then there is an angle path $\kappa: a \to a$ following $\alpha$ that does not form a full turn around $m$. 
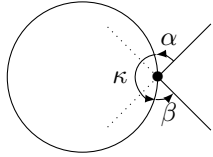
      \[
    \begin{tikzpicture}
    \draw (0,0) node {$\bullet$};
    \draw (-1,0) circle (1cm);
    \draw[-latex] (95:.3) arc (95:265:.3);
    \draw (-.5,0) node{$\kappa$};
    \draw (.15,.5) node{$\alpha$};
    \draw (.15,-.5) node{$\beta$};
    \draw (0,0)--(45:1);
    \draw[dotted] (0,0)--(135:1);
    \draw[dotted] (0,0)--(225:1);
    \draw (0,0)--(315:1);
    \draw[-latex] (45:.3) arc (45:95:.3);
    \draw[-latex] (265:.3) arc (265:315:.3);
    \end{tikzpicture} \]
    In that case $d\alpha$ can contain a term of the form $\kappa\alpha$. Let $\beta$ be the angle arrow that directly follows $\kappa$. Note that $\beta\ne\alpha$ otherwise the arc $a$ forms a monogon.
    Because $\beta$ and $\alpha$ are on different ends of the arc $a$, we have that 
    $$d(\beta\alpha)= \beta(c\kappa\alpha+\dots) \pm (d\beta)\alpha=0.$$
    Therefore $d(\beta)$ contains $c\beta\kappa$. The commutator $[c\kappa,-]$
    is only nonzero for the angle arrows $\alpha$ and $\beta$. By substracting this from $d$ we can make these terms disappear.
    \item If $a= t(\alpha)\ne h(\alpha)$ and $a$ is a loop we can do a similar reasoning as above.
    \item If $a= t(\alpha)\ne h(\alpha)=b$ and $a,b$ are both loops then $d\alpha$ can contain a term of the form $\kappa_1\alpha\kappa_2$. If $\kappa_1$ is not a full turn then let $\beta$ be the angle arrow such that $\beta\kappa_1\alpha\kappa_2\ne 0$. Because $\beta\alpha\ne 0$ we get
    $$d(\beta\alpha)= \beta(c\kappa_1\alpha\kappa_2+\dots) \pm (d\beta)\alpha=0.$$
    So $\kappa_2$ must end in $\alpha$ but then $\kappa_2$ is a full turn, so
    $\kappa_1\alpha\kappa_2= \ell_m^r\kappa_1\alpha$, and we can make this term disappear with $[\ell_m^r\kappa_1,-]$. If $\kappa_1$ is a full turn we can do the same.
    \item If neither $h(\alpha)$ or $t(\alpha)$ are loops then every path in $d(\alpha)$ is of the form $\ell_m^r\alpha$. 
    \end{enumerate}
From the discussion above we see that we can remove all terms that are not of the form $\ell_m^r\alpha$ by subtracting commutators. 
\end{proof}
\begin{remark}
Lemma \ref{center} holds more generally for marked surfaces with bouundary, but lemma \ref{derivations} does not hold in this generality (think of the cylinder with one marked point on each boundary circle, this has an arc collection whose gentle algebra is the Kronecker quiver $\circ\implies\circ$). However if we impose the [NL2] condition the lemma still holds. 
\end{remark}

We end this section with a nice interpretation of Koszul duality in this setting.
If $\cA$ is an arc collection on $(S,M)$, choose a set of points $F \subset S$ in the centers of the faces. For each arc $a$ draw a perpendicular arc $a^\perp$ that connects the centers of the faces adjacent to $a$ and points inside the left face. The \emph{dual arc collection} $\cA^\perp=\{a^\perp \mid a \in \cA\}$ forms an arc collection for $(S,F)$.   
\begin{theorem}
The Koszul dual of the gentle algebra is a gentle algebra of the dual arc collection:
\[
(\Gtl_\cA)^! \cong \Gtl_{\cA^\perp}. 
\]
\end{theorem}
\begin{proof}
From Bardzell \cite{bardzell1997alternating} and \cite{Bocklandt} we know that $A = \Gtl_\cA$ has a bimodule resolution $B^\bullet$ spanned by $A\otimes_\k b \otimes_\k A$
where $b = \beta_l\beta_{l-1} \cdots \beta_1 \in \C Q_{\cA}$ is a path of angle arrows such that all products $\beta_i\beta_{i-1}$ are zero. In other words paths that turn around faces. The maps between the terms have the following form
\[
 1\otimes b_k\dots b_1 \otimes 1 \mapsto b_k\otimes b_{k-1}\dots b_1 \otimes 1 - (-1)^k 1 \otimes b_k\dots b_{2}\otimes b_1
\]
Therefore the Koszul dual $\mathrm{Ext}^\bullet_A(\k,\k)=\H\Hom(B^\bullet, {}_L\k\otimes \k_R)$ is spanned by the dual basis $b^\vee$. The element $b^\vee$ corresponds to the (right) module extension
\[
0 \xleftarrow{}\C \id_{t(\beta_l)}\xleftarrow{\beta_l} \C \id_{h(\beta_l)}\xleftarrow{} \cdots\xleftarrow{} \C \id_{t(\beta_1)} \xleftarrow{\beta_1} ℂ\id_{h(\beta_1)}\xleftarrow{}0.
\]
Note we use right modules because in that way $\alpha$ and $\alpha^\vee$ run in opposite directions. Stitching together two of these module extensions shows that the ext product matches the (reverse) concatenation of paths if the paths turn around the same face. All other products are zero. In particular the ext algebra is spanned by dual angle arrows and
$\alpha^\vee\beta^\vee = 0$ if $m(\alpha)=m(\beta)$. Finally if $\alpha$ is an angle arrow then the dual $\alpha^\vee$ in the Koszul dual has degree $1-|\alpha|$. All this can be realised geometrically by considering the dual angles as angles between perpendicular arcs.
\begin{center}
    \begin{tikzpicture}
    \draw[black!50,-latex,very thick] (-30:2)--(0,0) node{$\bullet$};
    \draw[black!50,-latex,very thick] (30:2)--(0,0);
    \draw[black!50,-latex] (-30:.3) arc (-30:30:.3);
    \draw[black!50] (.7,0.1) node{$\alpha^\vee$};
    \draw[xshift=1.8cm,latex-,,very thick] (0,0)--(-120:1.5);    
    \draw[xshift=1.8cm,-latex,very thick] (0,0)node {$\bullet$}--(120:1.5);
    \draw (1.3,0) node{$\alpha$};
    \draw[xshift=1.8cm,-latex] (120:.3) arc (120:240:.3);
    \end{tikzpicture}
\end{center}
\end{proof}

\begin{remark}
In the theorem above we are referring to the classical Koszul dual \cite{BGS} where the gentle algebra is considered as a graded algebra graded by path length.
A similar statement holds if we work with Koszul duality for $A_\infty$-algebras but then we need to work with the gentle algebras completed by path length. 
\end{remark}

\subsection{Orbigons}

Now let $R$ be a complete local commutative ring over $ℂ$ with maximal ideal $R^+$ and residue field $ R/R^+ = ℂ $.
For each element $r \in Z(\Gtl_\cA) \htensor R^+ $ we will define a curved $R$-linear $A_\infty$-structure over $\Gtl_\cA\htensor R$. When tensored with $R/R^+$ the result will be the gentle $A_\infty$-algebra as defined in \cite{Bocklandt,HKK}, so our construction gives a family of deformations of the latter. 

To describe these deformations, we need the notion of an orbigon.
Morally this is a branched cover from a disk to the surface such that the boundary is mapped to arcs and the branch points are mapped to marked points. We can construct such branched covers by joining faces together. We will define this concept combinatorially and inductively in two steps.

\begin{definition}
\emph{Tree-gons} are certain sequences of angles up to cyclic permutation. The basic tree-gon come from faces: $(\alpha_k,\dots ,\alpha_1)$ is a tree-gon if they form the consecutive angles of a face such that $h(\alpha_i)=t(\alpha_{i+1})$. 

If $(\alpha_k,\dots ,\alpha_1)$ and $(\beta_l,\dots,\beta_1)$ are tree-gons such that $\alpha_1\beta_l\ne 0$ and $\alpha_k\beta_1\ne 0$ in $\Gtl_\cA$ then we define a new tree-gon 
\[
(\beta_1\alpha_k,\alpha_{k-1},\dots,\alpha_2,\alpha_1\beta_l,\dots,\beta_2)
\]
Geometrically this operation stitches the two tree-gons together over the common arc $h(\beta_l)=t(\alpha_1)$.
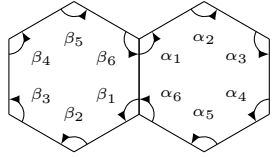
\begin{center}
    \begin{tikzpicture}
    \foreach \i in {1,...,6}
    {
    \draw (\i*60-90:1)--(\i*60-30:1);
    \draw (\i*60-90:.8) arc (\i*60+90:\i*60+90-60:.2);
    \draw[-latex] (\i*60-90:.8) arc (\i*60+90:\i*60+90+60:.2);    
    \draw (-\i*60+30:.5) node{{\tiny $\beta_\i$}};
    }
    \begin{scope}[xshift=1.73cm]
    \foreach \i in {1,...,6}
    {
    \draw (\i*60+30:1)--(\i*60+90:1);
    \draw (\i*60+30:.8) arc (\i*60+210:\i*60+210-60:.2);
    \draw[-latex] (\i*60+30:.8) arc (\i*60+210:\i*60+210+60:.2);
    \draw (-\i*60+210:.5) node{{\tiny $\alpha_\i$}};
    }
    \end{scope}
    \end{tikzpicture}
\end{center}
\end{definition}
\begin{remark}
From this definition one can deduce that tree-gons are sequences of internal angles of faces stitched together in a tree-like way. 
This tree, whose nodes are the faces and whose edges are the stitched arcs, can be reconstructed solely from the angle sequence.

Let $(\gamma_u,\dots,\gamma_1)$ be the sequence of all the indecomposable angle arrows in the tree-gon. Define a permutation $\sigma$ on $\{1,\dots,u\}$ as follows: for every $i$ there will be a unique shortest nontrivial path $\gamma_j\dots \gamma_{i+1}$ with $i+u \le j<i$ that lifts to a contractible loop in the universal cover of $S\setminus M$. Set $\sigma(i)=j \mod u$. One can easily show that for a tree-gon we have that $\sigma^2=\Id$ and $\sigma(i)=i$ if and only if $h(\gamma_i)$ is an arc on the boundary of the tree-gon. 
This means that the internal arcs are in 1--1 correspondence with the 2-cycles in $\sigma$. The nodes of the tree on the other hand correspond to the orbits of the permutation $\sigma \circ (u \dots 1)$. Indeed this permutation follows the angles and crosses over if the arc is internal, and therefore it cycles around the faces. An edge and a node are incident if their orbits intersect. 
\end{remark}

Now we will allow to fold together two consecutive arcs on the boundary of a tree-gon that are identical. The result is a disk-like shape with internal marked points, so some of the angles sit in the interior of the disk. We will use square brackets to denote what is interior. In general we get sequences of angles separated by square brackets and commas such as $(\alpha,\beta[\gamma[\delta]\epsilon]\zeta[\eta])$. The \emph{reduced sequence} cuts out anything that is between square brackets: e.g. $(\alpha,\beta\zeta)$. Again we work inductively. 

\begin{definition}
An \emph{orbigon} is a bracketed cyclic sequence of angles. The basic orbigons are tree-gons without any brackets and if $(\dots,U,\dots)$ is an orbigon for which the reduced sequence of $U$ is an angle that turns $r$ full turns around a marked point $m$ then $(\dots[U]\dots)$ is also an orbigon. The type of the orbigon is a multiset containing all pairs $p=(m,r)$ that were needed to introduce the brackets. Such a pair is also called an \emph{orbifold point}. 

Different tree-gons can be folded together to form the same disk, therefore we also need to impose an equivalence relation on the orbigons. 
Suppose $(\alpha_k,\dots ,\alpha_1)$ and $(\beta_l,\dots,\beta_1)$ are tree-gons and
\[
(\beta_1\alpha_k|\dots|\alpha_j\mathbf{[}\underbrace{\alpha_{j-1}|\dots |\alpha_1}_{U_1}\underbrace{\beta_l|\dots|\beta_{i+1}}_{U_2}\mathbf{]}\beta_i|\dots|\beta_2)
\]
is an orbigon, where the $|$-separators can be commas or brackets. Then the shift rule shifts the position of a piece between brackets and reverses the order of the two parts $U_1,U_2$ that are in thee different tree-gons as follows 
\[
(\beta_1\mathbf{[}\underbrace{\beta_l|\dots|\beta_{i+1}}_{U_2}\underbrace{\alpha_{j-1}|\dots |\alpha_1}_{U_1}\mathbf{]}\alpha_k|\dots|\alpha_j\beta_i|\dots|\beta_2)
\]
where brackets in $U_1$ that are paired up with a bracket $U_2$ and vice versa are flipped, for the formula to make sense.

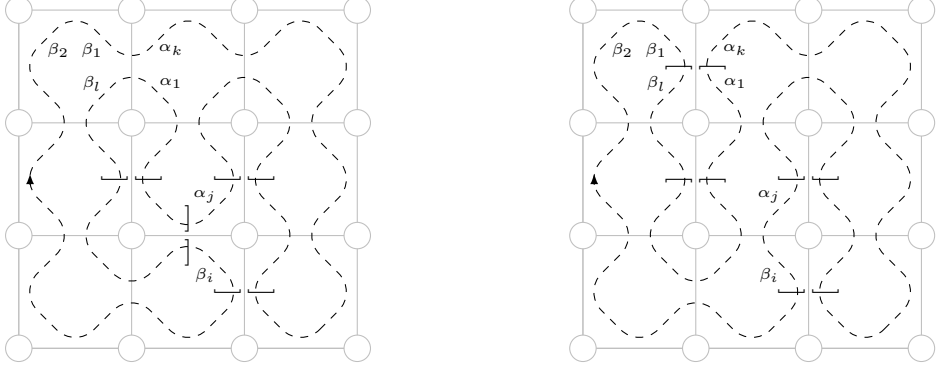
\begin{center}
    \begin{tikzpicture}
\begin{scope}[scale=1.5]
\draw[black!25] (0,0)--(3,0)--(3,3)--(0,3)--cycle;
\foreach \i in {0,1,2}
{\foreach \j in {0,1,2}
{
\draw[black!25] (\i+1,\j)--(\i,\j)--(\i,\j+1); 
}
}
\foreach \i in {0,1,2,3}
{\foreach \j in {0,1,2,3}
{
\node at (\i,\j) [black!25,circle,draw,fill=white,minimum size=10pt,inner sep=1pt] {};
}
}
 \draw[latex-,rounded corners=10,dashed] (2.5,0)--(3,.5)--(2.5,1)--(3,1.5)--(2.5,2)--(3,2.5)--(2.5,3)--(2,2.5)--(1.5,3)--(1,2.5)--(.5,3)--(0,2.5)--(.5,2)--(0,1.5)--(.5,1)--(0,.5)--(.5,0)--(1,.5)--(1.5,0)--(2,.5)--(1.5,1)--(1,.5)--(.5,1)--(1,1.5)--(.5,2)--(1,2.5)--(1.5,2)--(1,1.5)--(1.5,1)--(2,1.5)--(1.5,2)--(2,2.5)--(2.5,2)--(2,1.5)--(2.5,1)--(2,.5)--cycle;
\draw (1.35,2.65) node{{\tiny $\alpha_k$}};
\draw (1.35,2.35) node{{\tiny $\alpha_1$}};
\draw (.65,2.65) node{{\tiny $\beta_1$}};
\draw (.65,2.35) node{{\tiny $\beta_l$}};
\draw (.35,2.65) node{{\tiny $\beta_2$}};
\draw (1.65,.65) node{{\tiny $\beta_i$}};
\draw (1.65,1.35) node{{\tiny $\alpha_j$}};
\draw (1.5,.85)node{{$]$}};
\draw (1.5,1.15) node{{$]$}};
\draw[-latex] (0.1,1.45)--(0.1,1.55);

\draw (.85,1.5)node[rotate=90]{{$[$}};
\draw (1.15,1.5) node[rotate=90]{{$[$}};
\draw (1.85,1.5)node[rotate=90]{{$[$}};
\draw (2.15,1.5) node[rotate=90]{{$[$}};
\draw (1.85,.5)node[rotate=90]{{$[$}};
\draw (2.15,.5) node[rotate=90]{{$[$}};
\end{scope}

\begin{scope}[scale=1.5,xshift=5cm]
\draw[black!25] (0,0)--(3,0)--(3,3)--(0,3)--cycle;
\foreach \i in {0,1,2}
{\foreach \j in {0,1,2}
{
\draw[black!25] (\i+1,\j)--(\i,\j)--(\i,\j+1); 
}
}
\foreach \i in {0,1,2,3}
{\foreach \j in {0,1,2,3}
{
\node at (\i,\j) [black!25,circle,draw,fill=white,minimum size=10pt,inner sep=1pt] {};
}
}
 \draw[latex-,rounded corners=10,dashed] (2.5,0)--(3,.5)--(2.5,1)--(3,1.5)--(2.5,2)--(3,2.5)--(2.5,3)--(2,2.5)--(1.5,3)--(1,2.5)--(1.5,2)--(1,1.5)--(1.5,1)--(1,.5)--(.5,1)--(1,1.5)--(.5,2)--(1,2.5)--(.5,3)--(0,2.5)--(.5,2)--(0,1.5)--(.5,1)--(0,.5)--(.5,0)--(1,.5)--(1.5,0)--(2,.5)--(1.5,1)--(2,1.5)--(1.5,2)--(2,2.5)--(2.5,2)--(2,1.5)--(2.5,1)--(2,.5)--cycle;

\draw (1.35,2.65) node{{\tiny $\alpha_k$}};
\draw (1.35,2.35) node{{\tiny $\alpha_1$}};
\draw (.65,2.65) node{{\tiny $\beta_1$}};
\draw (.65,2.35) node{{\tiny $\beta_l$}};
\draw (.35,2.65) node{{\tiny $\beta_2$}};
\draw (1.65,.65) node{{\tiny $\beta_i$}};
\draw (1.65,1.35) node{{\tiny $\alpha_j$}};
\draw (.85,2.5)node[rotate=90]{{$]$}};
\draw (1.15,2.5) node[rotate=90]{{$]$}};
\draw[-latex] (0.1,1.45)--(0.1,1.55);

\draw (.85,1.5)node[rotate=90]{{$]$}};
\draw (1.15,1.5) node[rotate=90]{{$]$}};
\draw (1.85,1.5)node[rotate=90]{{$[$}};
\draw (2.15,1.5) node[rotate=90]{{$[$}};
\draw (1.85,.5)node[rotate=90]{{$[$}};
\draw (2.15,.5) node[rotate=90]{{$[$}};
\end{scope}
\end{tikzpicture}
\end{center}
The \emph{shift rule} generates an equivalence relation on the orbigons that preserves the type and the reduced sequence.
\end{definition}

\begin{remark}
Every tree-gon corresponds to a tree with nodes labeled by faces and edges labeled by arcs. Each folding operation adds another edge to the graph corresponding to the arc $a=h(U)$. In this way we end up with a graph of which this tree is a spanning tree. The graph, which we will sometimes refer to as the \emph{face graph}, is planar. It divides the plane into regions which can be labeled by a pair $(m,r)$ from the type multiset. Clearly the equivalence relation changes the underlying spanning tree while keeping the face graph the same. For the example above the face graph is a grid with $3\times 3$ nodes and the spanning trees are indicated in bold.
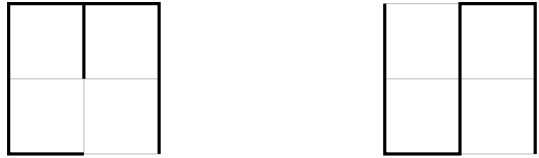
\begin{center}
\begin{tikzpicture}
\draw[black!25] (0,0)--(2,0)--(2,2)--(0,2)--cycle;
\draw[black!25] (1,0)--(1,2);
\draw[black!25] (0,1)--(2,1);
\draw[very thick] (2,0)--(2,2)--(0,2)--(0,0)--(1,0);
\draw[very thick] (1,2)--(1,1);

\begin{scope}[xshift=5cm]
\draw[black!25] (0,0)--(2,0)--(2,2)--(0,2)--cycle;
\draw[black!25] (1,0)--(1,2);
\draw[black!25] (0,1)--(2,1);
\draw[very thick] (2,0)--(2,2)--(1,2)--(1,0)--(0,0)--(0,2);
\end{scope}
\end{tikzpicture}    
\end{center}
The spanning tree of the face graph completely determines the unreduced sequences of the orbigon because it is the sequence of angles that runs around the tree. It also determines the brackets: each time you cross an edge of the face graph that is not in the spanning tree you open or close a bracket depending on whether you crossed it the first or the second time. 
\end{remark}
\begin{lemma}
Two orbigons are equivalent if and only if their face graphs are equivalent as labeled planar graphs.
\end{lemma}
\begin{proof}
By construction the shift rule does not change the underlying face graph. Now suppose that the two orbigons have isomorphic face graphs, then we have to show that we can move from one spanning tree $T_1$ to another spanning tree $T_2$ via the shift rules. 
If $a$ is an edge (or dually an arc) in $T_1$ not contained in $T_2$ then if we remove $a$, $T_1$ will split in two parts (or dually tree-gons) denote the angles in the first tree-gon by $\alpha_k,\dots,\alpha_1$ and those in the second tree-gon by $\beta_l,\dots,\beta_1$ such that $a=t(\beta_1)=h(\alpha_k)$. 

Because $T_2$ is connected it must contain an arc $b$ that connects the two parts of $T_1\setminus \{a\}$. Now find the indices $i,j$ such that $b=h(\beta_i)=t(\alpha_j)$ and perform the corresponding shift rule. 
The result will be an orbigon with a spanning tree equal to $T_1\setminus \{a\} \cup \{b\}$, which is one edge closer to $T_2$. Keep repeating this procedure until the spanning tree is $T_2$.
\end{proof}

\begin{remark}
Note that we can also stitch orbigons together over a common arc in their reduced sequences by stitching the underlying tree-gons and transferring the brackets to the stiched tree-gon. The face graph of the new orbigon consists of the two graphs of the smaller orbigons joined together by one edge labeled by the common arc. Moreover because the common arc must be contained in all spanning trees, this implies that the small orbigons are uniquely determined by the big one.
\end{remark}

\begin{lemma}\label{2possibilities}
If $(\alpha_k,\dots, \xi\eta, \dots \alpha_1)$ is the reduced sequence of an orbigon and $\xi,\eta$ are nontrivial angle paths then there are two possibilities:
\begin{itemize}
    \item[A] the orbigon is stitched together of two smaller orbigons with reduced sequences
    $(\beta_r, \dots, \xi)$ and $(\gamma_r, \dots, \eta)$ such that $\gamma_r\beta_r=\alpha_r$ for some $r$.
    \item[B] the orbigon is folded together from an orbigon with reduced sequence
    $(\alpha_k,\dots, \xi, \ell_m^r h(\eta),\eta, \dots, \alpha_1)$ 
\end{itemize}
in both cases these orbigons are unique.
\end{lemma}
\begin{proof}
Look at the face graph of the orbigon. The arc $t(\xi)=h(\eta)$ will correspond to an edge in this graph. If the graph remains connected after removing this edge then this new graph will be the graph of an orbigon that can be folded to the old orbigon and we are in situation B. Otherwise the old orbigon is stitched together from two smaller orbigons and we are in situation A.
\end{proof}

\begin{lemma}\label{finiteorbigons}
For a given type and reduced sequence there are at most a finite number of orbigons (up to equivalence). 
\end{lemma}
\begin{proof}
Each bracketing changes two commas in brackets and introduces an extra pair in the multiset of the type. Therefore the length of the unreduced sequence (the number of commas and brackets) is $2$ times the size of the multiset longer than the unreduced sequence (the number of commas). 
\end{proof}

\begin{lemma}\label{cycleconditions}
If the arc collection satisfies
\begin{itemize}
    \item the no monogon or digon condition [NMD] then all tree-gons have length at least $3$.
    \item the no loops or two-cycles condition [NL2] then all orbigons have reduced sequences of length at least 3.
\end{itemize}
\end{lemma}
\begin{proof}
If [NMD] holds then all faces are at least 3-gons. Gluing an $n$-gon to an $m$-gon results in an $n+m-2$-gon, and $n+m-2>2$ if $n,m>2$.
Furthermore, an orbigon with a reduced sequence of length one or two would give rise to a loop or two-cycle of arcs.
\end{proof}

\subsection{$A_\infty$-structures on the gentle algebra.}

Fix a commutative ring $R$.
Recall that if $A$ is a $\Z$ or $\Z/2\Z$-graded projective $R$-module then a curved $A_\infty$-structure is a collection of $R$-linear maps
\[
\mu^k : A^{\otimes_R i} \to A
\]
of degree $2-k$, satisfying the curved $A_\infty$-axioms:
\begin{equation*}
\sum_{k ≥ l ≥ m ≥ 0} (-1)^{‖x_m‖ + … + ‖x_1‖} μ(x_k, …, μ(x_l, …,x_{m-1}), x_m, …, x_1) = 0.
\end{equation*}
Here $‖x‖$ is shorthand for the \emph{shifted degree}: $‖x‖=|x| - 1$. If $l=m$ then we interpret the middle $\mu()$ as the element $\mu^0 := \mu^0(1) \in A$. This is called the curvature and if it is zero the structure is called uncurved.

If $A$ is an algebra, we say that $\mu$ is an \emph{extension of the product} if $a \cdot b = (-1)^{\deg b}\mu^2(a,b)$. If $A=\C Q/I \otimes R$ comes from a path algebra of a quiver we will take tensor products over $\k \otimes R$ instead of over $R$ and we ask that the vertex idempotents $\id_v$ are \emph{strict}: all  products $\mu^{\ne 2}$ for which one of the entries is such an idempotent are zero.

Now let $R$ be a local nilpotent or complete ring over $\C$ with maximal ideal $R^+$.
Fix an element $r \in Z(\Gtl_\cA) \htensor R^+$, where $ \htensor $ denotes the $ R^+ $-adic completed tensor product. We will write this element as
\[
r_0 + \sum_m r_m(\ell_m) 
\]
where $r_0 \in R^+$, the $r_m(t) \in R^+⟦t⟧$ are $ R^+ $-valued power series without constant term such that the reductions $ r_m ∈ (R^+ / (R^+)^k) ⟦t⟧ $ are actually polynomials for every $ k ≥ 1 $. We will write the $j^{th}$ coefficient of $r_m(t)$ as $r_p$ where $p=(m,j)$ is viewed as an orbifold point. We will also write $\ell_p$ as shorthand for $\ell_m^j$. We are now ready to define a family of extensions of the gentle algebra.

\begin{definition}
For any $r \in Z(\Gtl_\cA) \htensor R^+$ we define $\rmu_\bullet$ on $\Gtl_\cA\htensor R$ as follows
\begin{itemize}
\item For each orbifold point $p=(m,j)$ we define a nullary product
\[
\rmu^0_p: \k \to \Gtl_\cA\htensor R^+ : 1 \mapsto r_p\ell_m^j 
\]
\item  For each orbigon $\psi$ with reduced sequence $(\alpha_k,\dots,\alpha_1)$ of length $k$ we define a $k$-ary product $\rmu_\psi$.
If $\beta$ and $\gamma$ are angles such that $\beta\alpha_{k+i} \ne 0$ and $\alpha_i\gamma \ne 0$ then we set
\[
\rmu_\psi(\beta\alpha_{k+i},\dots,\alpha_i) = r_{\psi} \beta \text{ and }\rmu_\psi(\alpha_{k+i},\dots,\alpha_i\gamma) = (-1)^{|\gamma|} r_{\psi} \gamma.
\]
where $r_\psi$ is the product of all $r_p$ with $p$ running over the orbifold points in the type of $\psi$. If the type is empty (i.e. $\psi$ is a tree-gon) then we put $r_\psi=1$.
All other $\rmu_\psi(\beta_1,\dots,\beta_n)$ where the $\beta_i$ are angles are zero. 
\end{itemize}
The total product is the sum of all these products together with an overall curvature coming from $r_0$:
\begin{equation*}\begin{split}
\rmu^0 &= r_0 \mu^0_\id + \sum_p \rmu_p\\
\rmu^1 &= \sum_{|\psi|=1} \rmu_\psi\\
\rmu^2 &= \rmu^2_{\ord} + \sum_{|\psi|=2} \rmu_\psi\\
\rmu^k &= \sum_{|\psi|=k} \rmu_\psi.
\end{split}
\end{equation*}
where $\mu^0_\id$ is the nullary product with $\mu^0_\id(1)=1$ and $\rmu^2_{\ord}$ the $R$-linear extension of the ordinary product in $\Gtl_\cA$. This is well defined because by lemma \ref{finiteorbigons} there are only a finite number of $k$-orbigons for each type and $R$ is nilpotent or complete.
\end{definition}
\begin{remark}
The condition [NMD] implies that there are no tree-gons with length $1$ or $2$. This means that $\rmu^1 =0 \mod R^+$ and $\rmu^2= \mu^2_{\Gtl_{\cA}} \mod R^+$. 
The condition [NL2] implies that there are no orbigons with length $1$ or $2$. This means that $\rmu^1 =0$ and $\rmu^2= \rmu^2_{\ord}$. 
\end{remark}

\begin{remark}
If $r=0$, the only $\rmu_\psi$ that contribute are those coming from tree-gons and each contributes with a factor $r_\psi=1$. These are precisely those that correspond to immersed disks that do not cover marked points with their interior. Therefore the definition is equivalent to the definition in \cite{HKK}.
In \cite{Bocklandt} the definition of the higher product is given inductively, with an induction step that is analogous to the induction step we used to define a tree-gon. Hence the definition is also equivalent to the definition in \cite{Bocklandt}. 
From both papers we can conclude that $\rmu$ is an uncurved $A_\infty$-structure if $r=0$. If $r\ne 0$ then $\rmu$ will be a curved deformation of this uncurved $A_\infty$-structure. 
\end{remark}

\begin{remark}
Each higher product removes a sequence of angles that forms an orbigon. If the angle that remains comes from the first entry, we will call that the \emph{front of the product}, if it comes from the last we call it the \emph{back of the product}.
\end{remark}

\begin{proposition}\label{satisfies-curved-axioms}If [NL2] holds then
$\rmu$ is a well defined curved $A_\infty$-structure on $\Gtl_\cA \htensor R$, which is strict over $\k=\C^{\cA}$.
\end{proposition}
\begin{proof}
We need to check the curved $A_\infty$-axioms. The first two axioms 
\[
\rmu^1(\rmu^0(1))=0,~~\rmu^1(\rmu^1(\alpha) + \rmu^2(\rmu^0(1),\alpha) - \rmu^2(\alpha,\rmu^0(1))=0
\]
follow easily from the facts that $\rmu^0(1)$ is a central element and $\rmu^1=0$. 
The third axiom holds because $\rmu_{\ord}$ is associative and there are no $\rmu_3(\dots,\rmu_0(1),\dots)$ terms as this would imply an orbigon with reduced sequence $(\dots[\ell_m^k]\dots)$ of length $1$. This contradicts lemma \ref{cycleconditions}.

To show the higher axioms, first note that by construction all the $\rmu$ are strict over $\k$, so we only need to check the axioms when all entries are nontrivial paths. 
Fix an angle sequence $(\gamma_r,\dots, \gamma_1)$ of length $\ge 4$ and assume that there is a nonzero double product 
\[
\rmu_u(\gamma_{s+t-1},\dots,\gamma_{{i+t}},\rmu_v(\gamma_{i+t-1},\dots,\gamma_{i}),\gamma_{i-1},\dots,\gamma_1)
\]
where $\rmu_u$ can be a $\rmu_{\ord}$ (if $s=2$) or a $\rmu_\psi$ if $s>2$, and $\rmu_v$ can be a $\rmu_p$ (if $t=0$), a $\rmu_{\ord}$ (if $t=2$) or a $\rmu_\psi$ (if $t>2$). Such a product is uniquely characterized by the triple $(u,v,i)$.

Each nonzero double product is a path multiplied with a factor $\pm r_ur_v$ (where we use the convention that $r_\ord=1$). We will show there is a unique other triple with a nonzero double product cancel that cancels it. 

To do this, we will go through all possible triples $(u,v,i)$ systematically, making a distinction between whether $i$ is $1$, $s$  or somewhere in the middle. For each case we draw a diagram of the possible situations and every situation will occur exactly twice (see figure \ref{fig:cases}). Note that most diagrams have two versions depending on whether the (only/outer) higher product has a front or a back. We will only consider the latter cases.
The diagrams are named after the operations of type $A$, $B$ from lemma \ref{2possibilities} that can be performed on the orbigons to stitch or fold them together.

\begin{enumerate}
\item $(\psi,p,i)$ 
\begin{itemize}
\item If $i=1$ then $\ell_p$ contains the back of $\rmu_\psi$ and we are in situation $O1$. We can compensate this term with a term of a similar type but where $\ell_p$ contains the front:
\[
    \rmu_{\psi}(\gamma_s,\dots,\gamma_1, \rmu_p) \text{ cancels with } \rmu_{\psi}(\rmu_p,\gamma_s,\dots, \gamma_1)
    \]
 \item If $1<i<s$ then $\psi$ has  $\ell_p$ as one of its angles and we are in situation $B1$. We can fold the orbigon together to form a new orbigon $\psi'$ with $r_{\psi'}=r_\psi r_p$. The other term that cancels it is of the form $(\psi',\ord,i-1)$.
  
   \item If $i=s$ and $\ell_p$ comes before $\gamma_{s}$ then either $\gamma_s$ is shorter than the back of $\rmu_\psi$ or longer. If it is shorter, we are in situation $B2$. If it is longer than the back, the arc $t(\gamma_1)$ will lie inside $\psi$. In that case we distinguish type $BA$ if $t(\gamma_1)$ cuts $\psi$ in two pieces, or $BB$ if it opens a second fold at a marked point $q$. The former is compensated by a terms of type $(\psi_1,\psi_2,i)$, while the latter is compensated by a term $(\psi_1,q,i)$. In both cases $\psi_1$ includes the orbifold point $p$ and has a front)
 \end{itemize}
    \item $(\psi,\ord,i)$
\begin{itemize}
\item If $1<i\le s$ then the inner product $\rmu^2(\gamma_{i+1},\gamma_i)$ is an angle    
    from the orbigon $\psi$ then $h(\gamma_i)$ either opens a fold of $\psi$ ($B1$) or it cuts the orbigon in two $(A1L,A1R,A1M)$. The first is compensated by a term of type $(\psi',p,i+1)$, while the latter three by a term of type $(\psi_1,\psi_2,j)$. 
    \item    
    If $i=1$ and $\gamma_1$ is part of the back of $\rmu_\psi$ then $h(\gamma_1)$ will cut the back in two. This is situation $O2$ and it is compensated by a term of the form $(\ord,\psi,2)$. 
    
    If the back of $\rmu_\psi$ is part of $\gamma_1$ then $h(\gamma_1)$ either opens a fold of $\psi$ or it cuts the orbigon in two. Just like when $i\ne 1$ above, we are in cases ($B1,A1L,A1R,A1M$) but now with $\gamma_1$ equal to the top of the hexagon instead of the bottom. They are also compensated by terms of type $(\psi',p,i+1)$ and $(\psi_1,\psi_2,j)$.
\end{itemize}    
    \item $(\ord,\psi,i)$
\begin{itemize}
    \item If $i=1$ and $\rmu^2(\gamma_{t+1},\rmu_{\psi}(\dots,\gamma_1))$ puts something in front of the back of $\rmu_{\psi}$, three things can happen. If $\gamma_{t+1}$ is shorter than the first angle of $\psi$ then $h(\gamma_{t+1})$ either cuts $\psi$ in two $(A2)$ or opens a fold of $\psi$ ($B2$). These cases are compensated by terms of the form $(\psi_1,\psi_2,j)$ and $(\psi',p,t)$. 
    
    If $\gamma_{t+1}$ is longer than the first angle we are in situation $O3$ and can compensate it with a term of the same type but with a head
    \[
    \rmu^2(\rmu_{\psi}(\gamma_{t+1},\dots,\gamma_2),\gamma_1).
    \]
    \item If $i=2$
    then $\rmu^2(\rmu_{\psi}(\dots,\gamma_2),\gamma_1)$ adds something to the back of $\rmu_\psi$ and we are in situation $O2$.
\end{itemize}
    \item $(\psi_1,\psi_2,i)$
\begin{itemize}
    \item If $i=1$ and $\rmu_{\psi_2}$ has a back then we are in $A1M$.  If $\rmu_{\psi_2}$ has a front we are in situation $O4$. Then we can commute the order of $\psi_1$ and $\psi_2$, this gives two terms that cancel:
    \[ \rmu_{\psi_1}(\gamma_{s+t-1},\dots,\rmu_{\psi_2}(\dots,\gamma_1))
\text{ and }
    \rmu_{\psi_2}(\rmu_{\psi_1}(\gamma_{s+t-1},\dots),\dots,\gamma_1).
    \]
   \item If $1<i<s$ and $\rmu_{\psi_2}$ has a back equal to an angle of $\psi_1$ we are in situation $A1L$, while if it has a front equal to an angle, we are in situation $A1R$.
     \item If $i=s$ and $\rmu_{\psi_2}$ has a front we are in $A1R$.
    If it has a back then it depends on the back of $\rmu_{\psi_1}$. If it is at least as long as the last angle of $\psi_2$ we are in situation $A2$. Otherwise $t(\gamma_1)$ will cut $\psi_2$ in two pieces (situation $AA$) or it will open a fold ($AB$). The former is compensated by a term of type $(\psi_1',\psi_2',j)$ while the latter is compensated by a term $(\psi_1',p,j)$. In both cases the outer product now has a front. Note that type $AB$ and $BA$ only seem to occur once in the list but this is because they are canceled by terms with a front and these are not in the list.
\end{itemize}
 \end{enumerate}

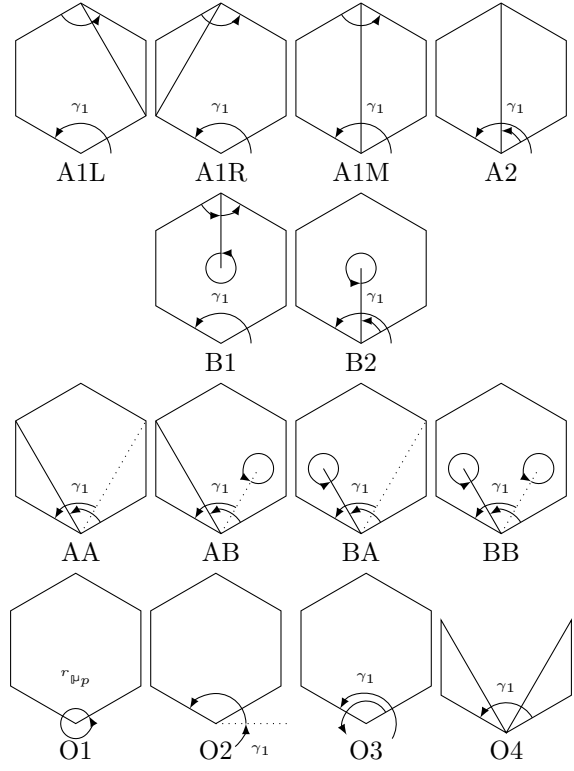
\begin{figure}
\begin{center}
\begin{tikzpicture}
\draw[yshift=1cm] (270:1)--(330:1)--(30:1)--(90:1)--(150:1)--(210:1)--cycle;
\draw[-latex] (0:.4) arc (0:150:.4); 
\draw[-latex,yshift=2cm] (210:.3) arc (210:330:.3);
\draw[yshift=1cm] (90:1)--(-30:1);
\draw(0,-.25) node{A1L};
\draw(0,.6) node{$\scriptscriptstyle{\gamma_1}$};
\end{tikzpicture}
\begin{tikzpicture}
\draw[yshift=1cm] (270:1)--(330:1)--(30:1)--(90:1)--(150:1)--(210:1)--cycle;
\draw[-latex] (0:.4) arc (0:150:.4); 
\draw[-latex,yshift=2cm] (210:.3) arc (210:330:.3);
\draw[yshift=1cm] (90:1)--(210:1);
\draw(0,-.25) node{A1R};
\draw(0,.6) node{$\scriptscriptstyle{\gamma_1}$};
\end{tikzpicture}
\begin{tikzpicture}
\draw[yshift=1cm] (270:1)--(330:1)--(30:1)--(90:1)--(150:1)--(210:1)--cycle;
\draw[-latex] (0:.4) arc (0:150:.4); 
\draw[-latex,yshift=2cm] (210:.3) arc (210:330:.3);
\draw[yshift=1cm] (90:1)--(270:1);
\draw(0,-.25) node{A1M};
\draw(.2,.6) node{$\scriptscriptstyle{\gamma_1}$};
\end{tikzpicture}
\begin{tikzpicture}
\draw[yshift=1cm] (270:1)--(330:1)--(30:1)--(90:1)--(150:1)--(210:1)--cycle;
\draw[-latex] (0:.4) arc (0:150:.4);
\draw[-latex] (30:.3) arc (30:90:.3);
\draw (0,0)--(0,2);
\draw(0,-.25) node{A2};
\draw(.2,.6) node{$\scriptscriptstyle{\gamma_1}$};

\end{tikzpicture}

\begin{tikzpicture}
\draw[yshift=1cm] (270:1)--(330:1)--(30:1)--(90:1)--(150:1)--(210:1)--cycle;
\draw[-latex] (0:.4) arc (0:150:.4); 
\draw[-latex,yshift=2cm] (210:.3) arc (210:270:.3);
\draw[-latex,yshift=2cm] (270:.3) arc (270:330:.3);
\draw[yshift=1cm] (90:1)--(0,0);
\draw[yshift=1cm,-latex] (90:.2) arc(90:450:.2);
\draw(0,-.25) node{B1};
\draw(0,.6) node{$\scriptscriptstyle{\gamma_1}$};

\end{tikzpicture}
\begin{tikzpicture}
\draw[yshift=1cm] (270:1)--(330:1)--(30:1)--(90:1)--(150:1)--(210:1)--cycle;
\draw[-latex] (0:.4) arc (0:150:.4);
\draw[-latex] (30:.3) arc (30:90:.3); 
\draw[yshift=1cm] (-90:1)--(0,0);
\draw[yshift=1cm,-latex] (-90:.2) arc (-90:270:.2);
\draw(0,-.25) node{B2};
\draw(.2,.6) node{$\scriptscriptstyle{\gamma_1}$};

\end{tikzpicture}

\begin{tikzpicture}
\draw[yshift=1cm] (270:1)--(330:1)--(30:1)--(90:1)--(150:1)--(210:1)--cycle;
\draw[yshift=1cm,dotted] (270:1)--(30:1);
\draw[yshift=1cm] (270:1)--(150:1);
\draw[-latex] (60:.4) arc (60:150:.4);
\draw[-latex] (30:.3) arc (30:120:.3);
\draw(0,.6) node{$\scriptscriptstyle{\gamma_1}$};
\draw(0,-.25) node{AA};
\end{tikzpicture}
\begin{tikzpicture}
\draw[yshift=1cm] (270:1)--(330:1)--(30:1)--(90:1)--(150:1)--(210:1)--cycle;
\draw[-latex] (60:.4) arc (60:150:.4); 
\draw[-latex] (30:.3) arc (30:120:.3);
\draw[yshift=1cm] (270:1)--(150:1);
\begin{scope}[rotate=-30]
\draw[yshift=1cm,dotted ] (-90:1)--(0,0);
\draw[yshift=1cm,-latex] (-90:.2) arc (-90:270:.2);
\end{scope}
\draw(0,-.25) node{AB};
\draw(0,.6) node{$\scriptscriptstyle{\gamma_1}$};
\end{tikzpicture}
\begin{tikzpicture}
\draw[yshift=1cm] (270:1)--(330:1)--(30:1)--(90:1)--(150:1)--(210:1)--cycle;
\draw[-latex] (60:.4) arc (60:150:.4); 
\draw[-latex] (30:.3) arc (30:120:.3);
\begin{scope}[rotate=30]
\draw[yshift=1cm] (-90:1)--(0,0);
\draw[yshift=1cm,-latex] (-90:.2) arc (-90:270:.2);
\end{scope}
\draw[yshift=1cm,dotted] (-90:1)--(30:1);
\draw(0,-.25) node{BA};
\draw(0,.6) node{$\scriptscriptstyle{\gamma_1}$};
\end{tikzpicture}
\begin{tikzpicture}
\draw[yshift=1cm] (270:1)--(330:1)--(30:1)--(90:1)--(150:1)--(210:1)--cycle;
\draw[-latex] (60:.4) arc (60:150:.4); 
\draw[-latex] (30:.3) arc (30:120:.3);
\begin{scope}[rotate=30]
\draw[yshift=1cm] (-90:1)--(0,0);
\draw[yshift=1cm,-latex] (-90:.2) arc (-90:270:.2);
\end{scope}
\begin{scope}[rotate=-30]
\draw[yshift=1cm,dotted ] (-90:1)--(0,0);
\draw[yshift=1cm,-latex] (-90:.2) arc (-90:270:.2);
\end{scope}
\draw(0,.6) node{$\scriptscriptstyle{\gamma_1}$};

\draw(0,-.25) node{BB};
\end{tikzpicture}

\begin{tikzpicture}
\draw[yshift=1cm] (270:1)--(330:1)--(30:1)--(90:1)--(150:1)--(210:1)--cycle;
\draw[-latex] (30:.2) arc (30:390:.2); 
\draw(0,-.35) node{O1};
\draw(0,.6) node{$\scriptscriptstyle{\rmu_p}$};
\end{tikzpicture}
\begin{tikzpicture}
\draw[yshift=1cm] (270:1)--(330:1)--(30:1)--(90:1)--(150:1)--(210:1)--cycle;
\draw[-latex] (-50:.4) arc (-50:0:.4);
\draw[-latex] (0:.4) arc (0:150:.4);
\draw[dotted] (0:0)--(0:1);
\draw(0,-.35) node{O2};
\draw(.6,-.35) node{$\scriptscriptstyle{\gamma_1}$};
\end{tikzpicture}
\begin{tikzpicture}
\draw[yshift=1cm] (270:1)--(330:1)--(30:1)--(90:1)--(150:1)--(210:1)--cycle;
\draw[-latex] (-30:.4) arc (-30:150:.4);
\draw[-latex] (30:.3) arc (30:210:.3);
\draw(0,-.35) node{O3};
\draw(0,.6) node{$\scriptscriptstyle{\gamma_1}$};
\end{tikzpicture}
\begin{tikzpicture}
\draw[yshift=1cm] (270:1)--(330:1)--(30:1)--(-90:1)--(150:1)--(210:1)--cycle;
\draw[-latex] (30:.4) arc (30:150:.4);
\draw(0,.6) node{$\scriptscriptstyle{\gamma_1}$};
\draw(0,-.23) node{O4};
\end{tikzpicture}
 \caption{The possible diagrams that can occur in as double products}
  \label{fig:cases}
\end{center}
\end{figure}
\end{proof}

\begin{remark}
The proof is analogous to the one in \cite{HKK}, but because of the internal orbifold points extra cases needed to be considered. These are all the cases with a $B$ in their label. 
The proof can also be extended to the case where [NL2] does not hold but this will include a lot more extra cases.
\end{remark}

\begin{remark}\label{Koszul1}
It is also possible to weight the faces in the orbigons using elements in $R^+$. Choose an element $s_{(f,i)} \in R^+$ for each face $f$ and $i \in \N$. The weight of a tree-gon is then defined inductively by giving $(\alpha_{ik},\dots,\alpha_1)$ weight $s_{(f,i)}$ if the angle sequence turns $i$ times around $f$ (note that this definition allows for orbifold faces as well). If you stitch two tree-gons together their weight is defined as the product. The folding operation adds an additional factor $r_{(m,j)}$ in the same way as before.

This gives a family of products ${}^{r,s}\bbmu^\bullet$ on $\Gtl_\cA \htensor R$ that depends on two parameters sets $r: M\times \N \to R^+$ and $s : M \times \N \to R^+$. Note that the curvature only depends on $r$ not on $s$.

This same parameter set can be used for the Koszul dual $\Gtl_{\cA^\perp}$ but now the role of $r$ and $s$ are reversed because for the Koszul dual the role of the marked points and the faces are swapped. In this case $s$ will give rise to curvature, while $r$ only contributes to the higher products. Note however that if $s_{(f,j)}$ is nonzero for an infinite number of $(f,j)$ we should go to the completed version ${\widehat{\Gtl}}_{\cA^\perp}$ for the curvature to make sense.
\end{remark}

\section{Hochschild Cohomology of Gentle $A_\infty$-algebras}

\subsection{Definitions}

\begin{definition}
If $(A,\mu)$ is a $\Z$- or $\Z_2$-graded $A_\infty$-algebra over a semisimple algebra $\k$, we define the $A_\infty$-Hochschild complex as
\[
 \HC^\bullet(A) = \Hom_\k(\bigoplus_{i\ge 0} A[1]^{ \otimes_\k i},A[1]) 
\]
Every element $\nu \in \HC^j(A)$ can be seen as a collection of $n$-ary products $\nu^n$ of degree $1+j-n$, one for each $n \in \Z_{\ge 0}$. We will call $\nu^n$ the $n$th component of $\nu$.
\end{definition}

\begin{remark}
Note that the grading on this space is not the classical grading on Hochschild cohomology coming from the number of entries in the products, but the one coming from the degrees of the maps $A[1]^{ \otimes_\k i} \to A[1]$. For this grading the product $\mu$ can be seen as an element of $\HC^1(A)$. To avoid confusion we will call this degree the \emph{$\infty$-degree} $ ‖\bullet‖ $. If we work $\Z_2$-graded we will also refer to this grading as the \emph{parity} and denote the homogeneous parts as  $\HC^{\even}(A)$ and  $\HC^{\odd}(A)$.
\end{remark}

\begin{itemize}
\item
On $\HC^\bullet(A)$ we have a bracket of degree $0$:
\begin{equation*}\begin{split}
[\kappa,\nu](a_r,\dots,a_1) :=
& \sum_{0\le i \le j \le r} (-1)^{(‖a_1‖ + … + ‖a_i‖)‖\nu‖}\kappa(a_r,\dots,a_{j+1},\nu(a_j,\dots,a_{i+1}),a_i,\dots, a_1)\\
&- (-1)^{‖ν‖‖κ‖ + (‖a_1‖ + … + ‖a_i‖)‖\kappa‖}\sum_{0\le i \le j \le r}\nu(a_r,\dots,a_{j+1},\kappa(a_j,\dots,a_{i+1}),a_i,\dots, a_1).
\end{split}\end{equation*}

The $A_\infty$-axioms for the product $\mu$, can be rephrased as $[\mu,\mu]=0$, which also means that $d =[\mu,-]$ is a differential of degree $1$ and the triplet
\[
(\HC^\bullet(A), d, [,])
\]
is a differential graded Lie algebra (DGLA). The solutions to the Maurer-Cartan equation
\[
d \nu +\frac{1}{2} [\nu,\nu]=0
\]
describe the deformations of $\mu$ as a curved $A_\infty$-structure.

\item
There is also a second product on $\HC^\bullet(A)$: the cup product.
\begin{equation*}\begin{split}
(\kappa \smile \nu) (a_r,\dots a_1) := \sum_{0\le i \le j \le u \le v \le r}
  (-1)^{\maltese}
 \mu(a_r,\dots,\kappa(a_v,\dots,a_{u+1}),\dots,\nu(a_j,\dots,a_{i+1}),\dots,a_1) 
\end{split}\end{equation*}

with $\maltese = {(‖a_1‖ + … + ‖a_u‖)‖\kappa‖ + (‖a_1‖ + … + ‖a_i‖)‖\nu‖ + ‖ν‖ + 1}$.
This product has degree $1$ and together with $d$ it satisfies the graded Leibniz rule. Therefore 
\[
(\HC^\bullet(A)[-1], d, \smile)
\]
is a differential graded algebra DGA and if we go to homology the triplet
\[
(\HH^\bullet(A)[-1], \smile, [,])
\]
is a Gerstenhaber algebra. 
\end{itemize}

\begin{remark}
These two classical structures on Hochschild cohomology of ordinary ungraded algebras were analyzed first by Gerstenhaber \cite{gerstenhaber1963cohomology}. Both definitions have been extended to $ A_∞ $-algebras and are well-known in the literature, see e.g.~\cite{mescher2016primer}.

Mescher \cite{mescher2016primer} uses a different sign convention for $ A_∞ $-categories. This means his definition of the cup product also has different signs. We have adapted the signs in order to suit the signs in our definition of $ A_∞ $-categories. In particular, the cup product as defined is graded symmetric, with respect to a degree shift of $ 1 $ on the Hochschild cohomology:
\begin{equation}
\label{eq:hochschild-cup-symmetry}
ν \smile η = (-1)^{(‖ν‖ - 1)(‖η‖ - 1)} η \smile ν.
\end{equation}
The unexpected shift by 1 in this sign rule is desired. In fact, this “shifted graded symmetry” renders the cup product truly graded symmetric with respect to the “traditional” grading of the Hochschild complex, which differs from the $ A_∞ $-grading precisely by one. In other words, if one regards an ordinary ungraded algebra and starts grading the Hochschild cohomology at zero (instead of minus one), the cup product becomes graded symmetric.

We have tried to arrange the signs such that the cup product together with the Gerstenhaber bracket turns $ \HH(\Gtl_{\cA}) $ into a Gerstenhaber algebra with correct signs. In order to be a Gerstenhaber algebra, sign conventions of cup product and bracket need to be tuned to each other. The relevant compatiblity condition is the signed Leibniz rule which reads
\begin{equation*}
[ν, η \smile ω] = [ν, η] \smile ω + (-1)^{‖ν‖ (‖η‖ - 1)} η \smile [ν, ω].
\end{equation*}
We have tried to arrange the signs with the Leibniz rule in mind, but we have not conducted the tedious checks. We have however chosen the signs in analogy with \cite{mescher2016primer}. In particular, the signs are such that the cup product descends to Hochschild cohomology and is graded symmetric, the latter we did check by hand.
\end{remark}

\begin{remark}\label{strictidempotents}
If $(A,\mu)$ is a strictly unital $A_\infty$-algebra over $\k$, i.e. $\mu(x_1,\dots,x_n)=0$ if $n\ne 2$ and at least one of the $x_i \in \k$, one can construct the \emph{normalized Hochschild cochain complex} $\underline{\HC}^\bullet(A)$. This is the subspace of all cochains that evaluate to zero if at least one of the entries is in $\k$. This subspace is closed under all the operations above the embedding is a quasi-isomorphism \cite{keller2006infinity}.
The Maurer-Cartan equation for $\underline{\HC}^\bullet(A)$ classifies deformations of $\mu$ for which $\k$ remains strict.
\end{remark}

We will now specialize to the case where $A=\Gtl_\cA$ is the gentle algebra of an arc collection $\cA$ equipped with $\mu=\rmu$ for $r=0$. We have seen that the $\Gtl_\cA$ has a natural $\Z_2$-grading, so the $\infty$-degree is a $\Z_2$-degree, which we refer to as the parity. 
It is possible to lift this $\Z_2$-grading to a $\Z$-grading. To do this we assign to each angle $\alpha$ a degree $\deg \alpha \in \Z$ with the same parity as $|\alpha|$. In order to make this degree compatible with $\mu$ we have to ensure that $\mu$ has $\infty$-degree $1$. This happens only when the total degree of the angles in a $k$-gon is $k-2$. As every angle occurs only in one $k$-gon this can always be done but not canonically.  

The gentle $A_\infty$-algebra also has an extra grading coming from the relative first homology 
\[
G =H_1(S\setminus M,\cA ,\Z).
\]
This group can be described in terms of the angles and faces
\[
G = \frac{\bigoplus_{\alpha \in (Q_{\cA})_1} \Z\alpha}{\langle\alpha_1+\dots+ \alpha_k \mid (\alpha_1,\dots,\alpha_k) \text{ is a face}\rangle}.
\]
The group $G$ comes with two natural maps 
\begin{itemize}
    \item $\pi : G \to \Z^{\cA}: \alpha \mapsto h(\alpha)-t(\alpha)$
    \item $\iota : \Z^{M} \to G: m \mapsto \sum_{\alpha \in \ell_m} \alpha$
\end{itemize}
We have $\pi\circ \iota=0$. The image of the $\pi$ has corank $1$, while the kernel of $\iota$ is $(1,\dots,1)$.  Therefore if $\#M\ge 2$ the map $\iota$ is nonzero. 

If we grade $\Gtl_\cA$ by giving each angle its corresponding degree in $G$, it is clear that all products $\mu$ have degree $0 \in G$. We can transfer the $G$-grading to the Hochschild complex and it is easy to see
that it contains only nonzero elements for degrees in $\Ker \pi$. Furthermore the natural operations $[,]$, $\smile$ and $d$ all have $G$-degree $0$.

\begin{remark}
It is also clear from this construction that if $\deg$ and $\deg'$ are two different $\Z$-lifts of the $\Z_2$-grading then their difference factors through $G$ because $\deg$ and $\deg'$ assign the same degree to a face. More precisely, the possible $\Z$-gradings form a $\Hom(G,2\Z)$-torsor.
\end{remark}

\begin{lemma}\label{totaldegreebelowzero}
If $\deg$ is a $\Z$-lift of the $\Z_2$-degree on $\Gtl_\cA$ then
\[
\sum_{m \in M} \deg \ell_m = 4-4g-2\#M = 2\chi(S,M) < 0.
\]
Here $g$ is the genus of the marked surface $(S,M)$ and $\chi(S,M)$ its Euler characteristic (which is negative by assumption).
\end{lemma}
\begin{proof}
\begin{equation*}\begin{split}
\sum_{m \in M} \deg \ell_m &= \sum_{\alpha \in (Q_{\cA})_1} \deg \alpha = 
\sum_{f \in F} \deg f \\&= \sum_{f \in F} (2- \#\{\alpha \in f\}) = 2\# F - 2\#\cA = 2(2-2g)-2\# M.
\end{split}\end{equation*}
\end{proof}

\subsection{Reduction to the zeroth and first component}
\label{reductiontozeroandfirstcomponent}
\newcommand{\eps}{\epsilon}
\newcommand{\CI}{\operatorname{Contr}}

In this section we will show that the Hochschild cohomology class of a cocycle can be read off from the 0th and 1st component. In other words, if $\nu^0$ and $\nu^1$ are both zero then $\nu$ is a coboundary. Remark \ref{strictidempotents} allows us to restrict to cochains which belong to the normalized Hochschild cohomology. For such classes we will construct an $\eps \in \underline{\HC}^\bullet(A)$ such that $\nu - d \eps$ evaluates zero on all sequences $(\beta_k,\dots, \beta_1)$ where the $\beta_i$ are nontrivial angle paths and $t(\beta_i)=h(\beta_{i+1})$. We will often refer to the procedure of adding a coboundary as \emph{gauging}.

\begin{definition}
Let $\beta = (\beta_k,\dots,\beta_1)$ be a sequence of nontrivial angles with $h(\beta_i)=t(\beta_{i+1})$ for $i<k$ (we do not impose that it cycles so  $h(\beta_i)$, $t(\beta_{1}$ can be different).
\begin{itemize}
\item We call $\beta$ \emph{elementary} if all $\beta_i$ are indecomposable angles (angle arrows). 
\item An index $i$ is called a \emph{contractible index} if $\beta_{i+1}\beta_i\ne 0$.
\item 
Denote the set of contractible indices by $\CI(\beta)$ and if $S \subset \CI(\beta)$ then the \emph{contracted sequence} $\beta^S$ will be the sequence where consecutive angles separated by a contractible index are multiplied together. 
\item
Every angle sequence is of the form $\beta^S$ where $\beta$ is elementary and $S$ a set of contractible indices. Moreover $\beta$ and $S$ are uniquely determined.
Two sequences are of the \emph{same type} if they are contractions of the same elementary sequence.
\item
The \emph{total length} of a sequence will be the length of its underlying indecomposable sequence.
\item
For an elementary sequence we have that $\CI(\beta)=\{\}$ if and only if it consists of consecutive angles of a polygonal face. Therefore we will call such sequences \emph{polygon sequences}.
\end{itemize}
\end{definition}

\begin{lemma}
\label{th:hochschild-odd-nu2}
Let $ ν ∈ \Ker(d) \subset \underline{\HC}^\bullet(\Gtl_{
\cA})$ with (A) $\nu^0 = ν^1 = 0 $. Then $ν$ can be gauged to satisfy additionally that (B) $ν^2(\alpha,\beta)=0$ for all pairs of angle paths with $\alpha\beta\ne 0$. During the gauging procedure, the value of $ ν $ does not change on polygon sequences.
\end{lemma}

\begin{proof}

We will construct an $ ε = ε^1$ such that $ν^2(\alpha,\beta)=(dε)^2(\alpha,\beta)$ if $\alpha\beta\ne 0$. Set $ ε^1 (α) = 0 $ for every indecomposable angle $ α $. Now define inductively $ ε^1 $ for decomposable angles (angle paths) by the rule
\begin{equation}
\label{eq:hochschild-odd-nu2-eps}
ε^1 (αβ) = ε^1 (α) β + α ε^1 (β) - (-1)^{|β|} ν^2 (α, β).
\end{equation}
Let us check that extending $ ε^1 $ according to this rule is well-defined, that is, $ ε^1 ((αβ)γ) = ε^1 (α(βγ)) $. Indeed,
\begin{align*}
& \big(ε^1 (αβ) γ + αβ ε^1 (γ) - (-1)^{|γ|} ν^2 (αβ, γ)\big) - \big(ε^1 (α) βγ + α ε^1 (βγ) - (-1)^{|βγ|} ν^2 (α, βγ)\big) \\
&= ε^1 (α) βγ + α ε^1 (β) γ - (-1)^{|β|} ν^2 (α, β) γ + αβ ε^1 (γ) - (-1)^{|γ|} ν^2 (αβ, γ) \\
& ~ - ε^1 (α) βγ - α ε^1 (β) γ - αβ ε^1 (γ) + (-1)^{|γ|} α ν^2 (β, γ) + (-1)^{|βγ|} ν^2 (α, βγ) \\
&= - (-1)^{|β|} ν^2 (α, β) γ - (-1)^{|γ|} ν^2 (αβ, γ) + (-1)^{|γ|} α ν^2 (β, γ) + (-1)^{|βγ|} ν^2 (α, βγ) \\
&= (-1)^{|β|} (dν) (α, β, γ) = 0.
\end{align*}
In the final row we have used that $ ν^0 = ν^1 = 0 $. We conclude that for angles $ α $, $ β $ with $\alpha\beta\ne 0$ we have
\begin{equation*}
(dε) (α, β) = (-1)^{|β|} ε^1 (α) β + (-1)^{|β|} α ε^1 (β) - (-1)^{|β|} ε^1 (αβ) = ν^2 (α, β)
\end{equation*}
as desired. Furthermore $ (dε)^0 = (dε)^1 = 0 $ and if $(α_k, …, α_1)$ is a polygon sequence, then
\begin{equation*}
(dε) (α_k, …, α_1) = ε (μ(α_k, …, α_1)) + μ(…, ε(α_i), …) = 0.
\end{equation*}
because all $\alpha_i$ are indecomposable so $ε(α_i)=0$.
\end{proof}

\begin{lemma}
Let $ ν ∈ \Ker(d) \subset \underline{\HC}^\bullet(\Gtl_{
\cA})$ with (A) $ \nu^0 = ν^1 =0$ and (B) $ν^2(\alpha,\beta)=0$ if $\alpha\beta\ne 0$. Then $ ν $ can be gauged to (C) evaluate zero on polygonal sequences without affecting the conditions (A), (B).
\end{lemma}

\begin{proof}
We start off with a little remark that is important to follow the argument.
Note that when going around a polygon $(\alpha_N,\dots,\alpha_1)$ the boundary arcs can be oriented clockwise or anticlockwise. If two consecutive arcs $a_i=t(\alpha_i)$ and $a_{i+1}=h(\alpha_i)$ have the same orientation then $|\alpha_i|=1$ and $‖\alpha_i‖=0$.
\begin{center}
\begin{tikzpicture}
\begin{scope}
  \draw[fill=black] (0,0) circle (.05cm);
\draw[thick,-latex] (225:1)--(0,0);
\draw[thick,-latex] (315:1)--(0,0);
\draw[-latex] (225:.5) arc (225:315:.5);
\draw (0,-1) node{$‖\alpha‖=1$};
\end{scope}

\begin{scope}[xshift=3cm]
  \draw[fill=black] (0,0) circle (.05cm);
\draw[thick,latex-] (225:1)--(0,0);
\draw[thick,latex-] (315:1)--(0,0);
\draw[-latex] (225:.5) arc (225:315:.5);
\draw (0,-1) node{$‖\alpha‖=1$};
\end{scope}

\begin{scope}[xshift=6cm]
  \draw[fill=black] (0,0) circle (.05cm);
\draw[thick,-latex] (225:1)--(0,0);
\draw[thick,latex-] (315:1)--(0,0);
\draw[-latex] (225:.5) arc (225:315:.5);
\draw (0,-1) node{$‖\alpha‖=0$};
\end{scope}

\begin{scope}[xshift=9cm]
  \draw[fill=black] (0,0) circle (.05cm);
\draw[thick,latex-] (225:1)--(0,0);
\draw[thick,-latex] (315:1)--(0,0);
\draw[-latex] (225:.5) arc (225:315:.5);
\draw (0,-1) node{$‖\alpha‖=0$};
\end{scope}
\end{tikzpicture}
\end{center}
Therefore if we have a polygonal sequence $(α_{i+k}, …, α_{i+1})$ then
the orientation of $h(α_{i+k})$ and $t(α_{i+1})$ will be the same (different) if the total shifted degree ${‖α_{i+k}‖ + … + ‖α_{i+1}‖}$ is zero (one). 

Working with the shifted degree is useful because the shifted degree of
$‖ν(α_{i+k}, …, α_{i+1})‖$ is the total shifted degree ${‖α_{i+k}‖ + … + ‖α_{i+1}‖}$ plus the parity of $\nu$. With this in mind we will
distinguish two cases depending on the parity of $\nu$.
\begin{itemize}
\item The parity of $\nu$ is odd.
We proceed by induction. Regard an elementary polygon $(\alpha_N,\dots,\alpha_1)$ of length $N$ and assume $ ν $ already vanishes on its polygon sequences of length $ ≤ k - 1 $. We will simultaneously gauge away $ ν $ on the polygon sequences of length $ k $ of this polygon. 




First note that an arc can occur at most twice on the boundary of a face and if it does it must be oriented once clockwise and once anticlockwise around the face because the surface is orientable. This implies that if $ ν(α_{i+k}, …, α_{i+1}) $ contains an identity, then $ k $ must be a multiple of $ N $. Indeed, if $k<N$ and there is an identity $\id_a$ in the result then  $t(\alpha_{i+1})=h(\alpha_{i+k})=a$. Because the two orientations of $a$ around the face are different the shifted degree of $(α_{i+k}, …, α_{i+1})$ is odd. As $\nu$ is odd, $ν(α_{i+k}, …, α_{i+1})$ will be even and hence it cannot contain $\id_a$ (whose shifted degree is odd).

Let us now assume $ k $ is not a multiple of $ N $. Then we can write $ ν(α_{i+k}, …, α_{i+1}) = α_{i+k} δ^{(i+k)} + γ^{(i+1)} α_{i+1} $. Note that
\begin{align*}
0 =& (dν) (α_{i+k+1}, …, α_{i+1}) \\
=&(-1)^{‖α_{i+k}‖ + … + ‖α_{i+1}‖} α_{i+k+1} (α_{i+k} δ^{(i+k)} + γ^{(i+1)} α_{i+1}) \\&+ (-1)^{‖α_{i+1}‖ + |α_{i+1}|} (α_{i+k+1} δ^{(i+k+1)} - γ^{(i+2)} α_{i+2}) α_{i+1} \\
=& α_{i+k+1} ((-1)^{‖α_{i+k}‖ + … + ‖α_{i+1}‖} γ^{(i+1)} - δ^{(i+k+1)}) α_{i+1}.
\end{align*}
In evaluating the first row we have used the induction hypothesis. Independent of arc directions, the angles $ α_{i+k+1} $ and $ γ^{(i+1)} $ are composable and $ δ^{(i+k+1)} $ and $ α_{i+1} $ are composable. We deduce $ δ^{(i+k+1)} = (-1)^{‖α_{i+k}‖ + … + ‖α_{i+1}‖} γ^{(i+1)} $ for all $ i $. Set $ ε (α_{i+k-1}, …, α_{i+1}) = (-1)^{‖α_{i+k-1}‖ + … + ‖α_{i+1}‖ + 1} δ^{(i+k)} $ for all $ i $. Then
\begin{align*}
(dε) (α_{i+k}, …, α_{i+1}) &= α_{i+k} δ^{(i+k)} + (-1)^{‖α_{i+k}‖ + … + ‖α_{i+2}‖ + 1 + |α_{i+1}|} δ^{(i+k+1)} α_{i+1} \\
&= α_{i+k} δ^{(i+k)} + γ^{(i+1)} α_{i+1} = ν(α_{i+k}, …, α_{i+1}).
\end{align*}
We have $ (dε)^0 = (dε)^1 = 0 $. Let us now check that $ dε $ does not affect any other elementary polygon sequences or the sequences in the same polygon with length $ l ≤ k-1 $. Regarding shorter sequences in the same polygon, we have
\begin{equation*}
(dε) (α_{i+l}, …, α_{i+1}) = μ(…, ε, …) + ε (…, μ(α_{i+s+tN}, …, α_{i+s+1}), …)
\end{equation*}
where $ N $ is the length of the polygon. The first summand vanishes, since $ l ≤ k - 1 $. In the second summand, the inner $ μ $ may only yield an identity and $ ε $ vanishes. For sequences in other polygons, similar arguments apply. We have safeguarded that $\nu^{0,1}=0$ is preserved when gauging by $ ε $. For $ k > 2 $, also $ (dε)^2 $ vanishes. In case $ k = 2 $ the $(d\eps)^2(\alpha,\beta)$ may be nonzero in rare cases, but we fix this by applying \autoref{th:hochschild-odd-nu2}, without changing $ ν $ on polygon sequences.


Finally, let us treat the case where $ k $ is a multiple of the length $ N $ of the polygon. Apart from the identities, we can gauge everything in $ ν(α_{i+k}, …, α_{i+1}) $ away as above. It remains to check out the identities. Write $ ν(α_{i+k}, …, α_{i+1}) = c_i \id_{a_i} $. Note we have
\begin{align*}
0 &= (dν) (α_{i+k+1}, …, α_{i+1}) \\
&= α_{i+k+1} ν(α_{i+k}, …, α_{i+1}) - ν(α_{i+k+1}, …, α_{i+2}) α_{i+1} \\
&= c_i α_{i+k+1} - c_{i+1} α_{i+1} = (c_i - c_{i+1}) α_{i+1}.
\end{align*}
Here we have used that $ ν $ already vanishes on sequences of the same polygon of length $ ≤ k-1 $. We obtain that all $ c_i $ around the polygon are equal. Denote this value by $ c $. Choose an angle $ α_1 $ in the polygon and define $ ε(α_{k-N+1}, …, α_1) ≔ c α_1 $. As desired,
\begin{align*}
(dε) (α_{i+k}, …, α_{i+1}) &= μ(…, ε(α_{k+1}, …, α_1), …) \\
&= c μ(α_{i+N}, …, α_{i+1}) = c \id_{a_i}.
\end{align*}
In the first row we have used that the sequence $ α_1, …, α_{k-N+1} $ of length $ k-N+1 $ appears precisely once in the sequence $ α_{i+k}, …, α_{i+1} $ and hence an inner $ ε $ can be applied precisely once. We see that $ (dε)^0 = (dε)^1 = 0 $. Let us check that $ dε $ vanishes on any polygon sequence $ β_1, …, β_l $ of length $ l ≤ k-1 $ in the same polygon. Indeed,
\begin{equation*}
(dε) (β_l, …, β_1) = ε(…, μ(…), …) + μ^2 (…, ε(…), …) + μ^{≥3} (…, ε(…), …).
\end{equation*}
The first summand vanishes, because the inner $ μ $ only gives identities. The second summand vanishes, because the result of the inner $ ε $ is a multiple of the first, equivalently last input of its input sequence, hence not composable with the input angle of the outer $ μ^2 $. The third summand vanishes, because $ ε $ consumes $ k-N+1 $ inputs and the outer $ μ^{≥3} $ needs $ N-1 $ more inputs, while the sequence $ β_1, …, β_l $ has length only $ l ≤ k-1 $. Similarly, one checks that $ dε $ vanishes entirely on polygon sequences in other polygons. For $ k > N $, we have $ (dε)^2 = 0 $. For $ k = N $, it may happen that $ (dε)^2 ≠ 0 $, but we fix this by applying \autoref{th:hochschild-odd-nu2}.

In total, we have gauged $ ν $ infinitely many times during this proof. However, the $ ε $ gauges have higher and higher input length. Moreover we only invoke \autoref{th:hochschild-odd-nu2} finitely many times, this means that the total sum of the gauges is defined in $\Hom_\k(\bigoplus_{i\ge 0} A[1]^{ \otimes_\k i},A[1])$ and we conclude it is a Hochschild cochain.

\item The parity of $\nu$ is even.
As in the odd case, we proceed again by induction over the length $ k $ of the polygon sequence. 

Let us first check for possible identities in $ ν(α_{i+k}, …, α_{i+1}) $. In case $ k $ is a multiple of $ N $, the source arc of $ α_{i+1} $ is equal to the target arc of $ α_{i+k} $, in particular $||ν(α_{i+k}, …, α_{i+1})||$ is even and does not contain identities. In case $ k $ is not a multiple of $ N $, we observe
\begin{equation*}
(dν) (α_{i+k+1}, α_{i+k}, …, α_{i+1}) = α_{i+k+1} ν(α_{i+k}, …, α_{i+1}) + ν(α_{i+k+1}, …, α_{i+2}) α_{i+1}.
\end{equation*}
Since $ k $ is not a multiple of $ N $, we have $ α_{i+k+1} ≠ α_{i+1} $ and conclude that $ ν(α_{i+k}, …, α_{i+1}) $ contains no identities.

We now proceed with gauging $ ν(α_{i+k}, …, α_{i+1}) $ to zero. By abuse of wording, let us say “head” and “tail” of a polygon's arc $ a $ to mean the vertex at the clockwise end and at the counterclockwise end of $ a $. For every $ i $ we write
\begin{equation*}
ν(α_{i+k}, …, α_{i+1}) = κ_i + λ_i,
\end{equation*}
where $ κ_i $ contains the angle paths that starts at the “head” of $ a_{i+1} $ and $ λ_i $ starts at the “tail” of $ a_{i+1}=t(α_{i+1})$. 
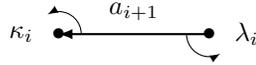
\begin{center}
\begin{tikzpicture}
\draw[thick,latex-] (0,0) node{$\bullet$}--(2,0) node{$\bullet$};
\draw (1,.3) node{$a_{i+1}$};
\draw[-latex] (0:.3) arc (0:120:.3);
\draw[-latex,xshift=2cm] (180:.3) arc (180:300:.3);
\draw (-.5,0) node{$\kappa_i$};
\draw (2.5,0) node{$\lambda_i$};
\end{tikzpicture}
\end{center}
We show that $ λ_i $ necessarily vanishes. Regard the source arc $ a_{i+1} $ of $ α_{i+1} $ and target arc $ a_{i+k+1} $ of $ α_{i+k} $. We have
\begin{equation*}
0 = (dν) (α_{i+k+1}, …, α_{i+1}) = α_{i+k+1} (κ_i + λ_i) + (κ_{i+1} + λ_{i+1}) α_{i+1}.
\end{equation*}
We know $ κ_i $ is odd or even, depending on whether $ a_{i+1} $ and $ a_{i+k+1} $ are equally oriented with respect to the polygon or not. Since $ κ_i $ always starts at the “head” of $ a_{i+1} $, we deduce from its degree that it always ends at the “tail” of $ a_{i+k+1} $. In particular $ α_{i+k+1} κ_i = 0 $. Similarly, $ λ_i $ starts at the “tail” of $ a_{i+1} $, hence ends at the “head” of $ a_{i+k+1} $, whether $ a_{i+1} $ and $ a_{i+k+1} $ are oriented equally or not. In particular $ α_{i+k+1} $ and $ λ_i $ are composable. But $ α_{i+1} $ starts at the “head” of $ a_{i+1} $, while $ λ_i $ starts at the “tail” of $ a_{i+1} $, hence the summands $ α_{i+k+1} λ_i $ and $ (κ_{i+1} + λ_{i+1}) α_{i+1} $ consist of disjoint sets of angles. Since the whole sum is supposed to vanish, we conclude $ λ_i = 0 $.

Since $ κ_i $ starts at the side of $ α_{i+1} $ and $ ν(α_{i+k}, …, α_{i+1})$ contains no identities, we can write $ κ_i = γ_i α_{i+1} $. We aim at gauging $ ν(α_{i+k}, …, α_{i+1}) $ to zero. Let us first gauge away all terms except the possible scalar multiple of $ α_{i+1} $, which comes from a possible identity in $ γ_i $. Put $ ε(α_{i+k}, …, α_{i+2}) ≔ γ_i $. Then
\begin{equation*}
(dε) (α_{i+k}, …, α_{i+1}) = γ_i α_{i+1} + α_{i+k} γ_{i-1}.
\end{equation*}
Apart from identities in $ γ_{i-1} $, the angles $ α_{i+k} $ and $ γ_{i-1} $ are not composable. To see this, consider two cases. If $ a_i $ and $ a_{i+k} $ are oriented opposite, then $ ν(α_{i+k-1}, …, α_i) = κ_{i-1} $ is even. Hence $ γ_{i-1} $ enters $ a_{i+k} $ at the opposite side of where $ α_{i+k} $ leaves. If $ a_i $ and $ a_{i+k} $ are oriented equally, then $ κ_{i-1} $ is odd and $ γ_{i-1} $ enters $ a_{i+k} $ still at the opposite side of where $ α_{i+k} $ leaves. Either way, we have $ α_{i+k} γ_{i-1} = 0 $ and $ γ_i α_{i+1} = κ_i $ remains as desired.

Let us check that $ dε $ vanishes on polygon sequences in other polygons and on polygon sequences shorter than $ k $ in the same polygon. Indeed,
\begin{equation*}
(dε) (β_l, …, β_1) = ε(…, μ(…), …) + μ(…, ε(…), …).
\end{equation*}
In the first summand, the inner $ μ $ can only give identities, on which $ ε $ vanishes. In the second summand, the input sequence of the inner $ ε $ must be from the same polygon as $ α_1, …, α_N $ and of length precisely $ k-1 $. Since the outer $ μ $ is $ μ^{≥2} $, this means that $ β_1, …, β_l $ has length at least $ k $, which was not to be assumed.

Finally, let us gauge away the remaining scalar multiples of $ α_{i+1} $ in $ ν(α_{i+k}, …, α_{i+1}) $. For $ i ∈ ℤ/Nℤ $, write $ ν(α_{i+k}, …, α_{i+1}) = c_i α_{i+1} $. If $ c_i ≠ 0 $, we deduce that the target arc of $ α_{i+k} $ is equal to the target arc of $ α_{i+1} $, that is, $ a_{i+k+1} = a_{i+2} $. Moreover, $ ν $ is even and hence $ a_{i+k+1} $ and $ a_{i+2} $ are oriented equally. In other words, $ a_{i+k+1} $ and $ a_{i+2} $ are equal as arcs and are oriented equally in the polygon. As we remarked in the odd case, this is only possible if $ i+k+1 = i+2 $ in $ ℤ/Nℤ $. We conclude that if some $ c_i $ does not vanish, then $ k = lN + 1 $ for some $ l ≥ 1 $.

Assuming $ k = lN + 1 $, let us show that the different $ c_i $ along the polygon are related. Indeed,
\begin{equation*}
0 = (dν) (α_{(l+1)N}, …, α_1) = \sum_{i = 1}^N c_i \id_{a_1} + μ^2 (α_{(l+1)N}, ν(…)) + μ^2 (ν(…), α_1).
\end{equation*}
We have used that an inner $ ν $ can be placed in precisely $ N $ ways, replicating the corresponding angle $ α_i $ and forming a disk together with the remaining angles. Moreover, an inner $ μ $ is impossible, since it only yields identities. An outer $ μ^2 $ is possible, but gives non-empty angles only which can be distinguished from the identities. We deduce that the sum of all $ c_i $ vanishes. Put 
\begin{equation*}
ε(α_{i+lN}, …, α_{i+1}) ≔ - \left(\sum_{j = 1}^{i-1} c_j\right) \id_{a_{i+1}}.
\end{equation*}
Then
\begin{equation*}
(dε) (α_{i+lN+1}, …, α_{i+1}) = - \sum_{j = 1}^{i-1} c_j α_{i+lN+1} + \sum_{j = 1}^i c_j α_{i+1}.
\end{equation*}
This definition makes sense for $ i ∈ ℤ/Nℤ $, since the sum over $ c_i $ vanishes. The sums cancel each other because $ α_{i+lN+1} = α_{i+1} $, and $ c_i α_{i+1} $ remains as desired. It is standard to check that $ dε $ does not have values on shorter polygon sequences or other polygons.
\end{itemize}
\end{proof}

\begin{lemma}
\label{th:hochschild-odd-compositions}
Let $ ν ∈ \Ker(d) \subset \underline{\HC}^\bullet(\Gtl_{
\cA}$ with (A) $ν^0 = ν^1 = 0$, (B) $\nu(\alpha,\beta)=0$ if $\alpha\beta\ne 0$ and suppose (C) $\nu$ vanishes on polygonal sequences. Then $ ν $ can be gauged to zero on all angle sequences.
\end{lemma}

\begin{proof}
We prove the statement by gauging $ ν $ to zero inductively on the total length of the angle sequences. After each step of gauging, we prove that all additional assumptions (A-C) still hold. 

We already know $ ν^{1} = 0$ and if $(\alpha,\beta)$ is a sequence of length two then either $\alpha\beta \ne 0$  of $(\alpha,\beta)$ is a polygon sequence, so by (B) and (C) we are done for total length at most $2$.

Now let $(α_k, …, α_1) $ be any sequence of $ k ≥ 3 $ angles $ α_i: a_i → a_{i+1} $. By induction, we assume that $ ν $ already vanishes on all shorter sequences. Since $ν$ already vanishes on polygon sequences, we can assume $(α_k, …, α_1)$ is not a polygon sequence.

We will gauge $ ν $ on $ (α_k, …, α_1)$ simultaneously with all other sequences of the same type and therefore we will assume $(α_k, …, α_1)$ is elementary. Now $ (α_k, …, α_1)$ consists of indecomposable angles of which at least one pair is composable (otherwise it would be a polygon sequence). Orientations of arcs will play no role in this proof, apart from determining the degrees of the angles. For simplicity, let us assume all $α_i$ have odd degree. (This happens in case of a dimer.) Similarly, $ ν $ is either of odd or even parity, and the signs we write are for the odd case.

Let $\CI(\beta)$ be the set of all contractible indices and $s = \max(\CI(\beta))$. We will now gauge all $ β^S $ for $ S ⊂ \CI(\beta)$ simultaneously. Put
\begin{equation*}
ε(β^{\{s\} \cup T}) = (-1)^{|T|} ν(β^T), \quad T ⊂ \CI(\beta)\setminus\{s\}.
\end{equation*}
We show that $ (dε)(β^S) = ν(β^S) $ for all $ S ⊂ \CI(\beta)$. 
\begin{itemize}
\item
In case $ s \notin S $, we have
\begin{equation*}
(dε) (β^S) = ε(β^{\{s\} \cup S}) = (-1)^{|S|} ν(β^S).
\end{equation*}
by definition. Note that any $ ε(…, μ^{≥3}, …) $ and $ μ(…, ε, …) $ vanish since their input sequence of $ ε $ is shorter than $(α_k, …, α_1)$. 
\item
In case $ s ∈ S $, first observe that for $ K ⊂ \CI(\beta) $ we have
\begin{equation*}
\sum_{t \notin K} (-1)^{1 + |K_{<t}|} ν(β^{K \cup \{t\}}) = (dν) (β^K) ∓ ν(…, μ^{≥3}, …) ∓ μ(…, ν, …) = 0.
\end{equation*}
Here $ |K_{<t}| $ denotes the number of indices in $ K $ smaller than $ t $. Indeed, the second and third summands vanish because the input sequence of $ ν $ is shorter than $ β $. Now using this observation for $ K = S \setminus \{s\} $ we get
\begin{align*}
(dε) (β^S) &= - \sum_{t \notin S} (-1)^{1 + |S_{<t}|} ε(β^{S \cup \{t\}}) \\
&= \sum_{t \notin S} (-1)^{|S \setminus \{s\} \cup \{t\}| + |S_{<t}|} ν(β^{S \setminus \{s\} \cup \{t\}}) \\
&= \sum_{t \notin S \setminus \{s\}} (-1)^{|S| + 1} (-1)^{|(S \setminus \{s\})_{<t}| + 1} ν(β^{S \setminus \{s\} \cup \{t\}}) - (-1)^{|S| + 1} (-1)^{|S_{<s}| + 1} ν(β^{S \setminus \{s\} \cup \{s\}}) \\
&= 0 + ν(β^S)
\end{align*}
as desired. In the first row, we have used that all $ ε(…, μ^{≥3}, …) $ and $ μ(…, ε, …) $ vanish since their input sequence of $ ε $ is shorter than $ (α_k, …, α_1) $. 
\end{itemize}
Finally, $ (dε)^0 = (dε)^1 = 0 $ and $ dε $ vanishes on polygon sequences $ γ_l, …, γ_1 $. Indeed,
\begin{equation}
\label{eq:hochschild-odd-distribution-deps-vanishing}
(dε) (γ_l, …, γ_1) = ε(…, μ, …) + μ(…, ε, …).
\end{equation}
The first summand vanishes because the inner $ μ $ yields only vertex idempotents. The second summand vanishes, since the input sequence of $ ε $ consists only of indecomposable angles, while a sequence $ β^{\{s\} ∪ T} $ on which $ ε $ is defined has at least one decomposable angle.

Next, note that $ dε $ vanishes if an input is a vertex idempotent. Now assume $ γ_l, …, γ_1 $ is any sequence of non-empty angles of total length less than or equal to that of $ α_k, …, α_1 $. Consider the sum \eqref{eq:hochschild-odd-distribution-deps-vanishing} again. The first summand vanishes in case of $ μ^{≥3} $ because then the input sequence of $ ε $ has less total length than $ α_1, …, α_k $. If the first summand does not vanish for some $ μ^2 $, we conclude that the indecomposable constituents of the sequence are equal to $ α_1, …, α_k $. In other words, it is one of those sequences we have just gauged correctly already. The second summand vanishes, since the input sequence of $ ε $ is shorter than $ α_k, …, α_1 $.
\end{proof}

\begin{theorem}\label{lowdegzeroiszero}
Let $ ν ∈ \Ker(d)$ with $\nu^{0,1}=0$ then $\nu$ is zero in $\HH^\bullet(\Gtl_\cA)$.
\end{theorem}
\begin{proof}
This is an immediate consequence of the previous lemmas.
\end{proof}

This theorem implies that the cohomology class of the cocycles can be read off from its zeroth and first components.  
\begin{lemma}\label{central-and-derivation}
If $\nu$ is a cocycle then 
\begin{itemize}
    \item $\nu^0$ is a central element,
    \item if $\nu^0=0$ then $\nu^1$ is a derivation.
\end{itemize}
\end{lemma}
\begin{proof}
We have that $ν^0(1)$ is a central element because 
\begin{equation*}
0 = (dν)^1 (α) =  μ^2 (α, ν^0(1)) + (-1)^{‖α‖} μ^2 (ν^0(1), α).
\end{equation*}
If $\nu^0=0$ then $\nu^1$ is a derivation because 
\begin{align*}
0 &= (-1)^{|β|} (dν)(α, β) = ν^1 (α) β + α ν^1 (β) - ν(αβ), \\
0 &= ν^1 (\id_a) \id_a + \id_a ν^1 (\id_a) - ν(\id_a).
\end{align*}
\end{proof}

Depending on the parity of the cocycle we can be even more precise: odd cocycles are determined by $\nu^0$, which is a central element, while even cocycles are determined by $\nu^1$, viewed as an outer derivation of $\Gtl_\cA$.  
\begin{theorem}\label{embeddings}
Let $\cA$ be an arc collection.
\begin{enumerate}
    \item The map $$\HH^{\even}(\Gtl_\cA) \to \mathrm{OutDer}(\Gtl_\cA): \nu \mapsto \nu^1$$ is a well-defined embedding.
    \item The map $$\HH^{\odd}(\Gtl_\cA) \to Z(\Gtl_\cA): \nu \mapsto \nu^0(1)$$ is a well-defined  embedding.
\end{enumerate}
\end{theorem}
\begin{proof}
For the first statement note that the center only has cycles of even degree, which represent odd maps $A[1]^{\otimes 0}\to A[1]$. Therefore $\nu^0=0$ if the parity of $\nu$ is even and hence by the previous lemma $\nu^1$ is indeed a derivation. The map is also well defined because
\begin{align*}
(d\kappa)^1(\alpha) &= \mu^1(\kappa^1(\alpha))+ \mu^2(\kappa^0,\alpha) - \mu^2(\alpha,\kappa^0)=[\kappa^0,\alpha].
\end{align*}
Furthermore if $\nu^1$ is inner and $\nu^0=0$ then we can find a $\kappa$ such that $(\nu-d\kappa)^{1}=0$ zero. Because $(\nu-d\kappa)^0$ is zero for degree reasons we have by theorem \ref{lowdegzeroiszero} that $\nu=0$ in $\HH^{\even}(\Gtl_\cA)$. Therefore the map is an embedding.

For the second statement, first note that this map is well-defined because
$(d\kappa)^0 = \mu^1\circ\kappa^0=0$. Now suppose that $\nu^0(1)=0$, then we will construct an even $ ε $ such that $ (dε)^1 (α) = ν^1 (α) $ for all indecomposable angles. Because $\nu^0(1)=(dε)^0=0$ both $(dε)^1$ and $\nu^1$ are derivations and therefore $(dε)^1=\nu^1$. This means that $(\nu-d\epsilon)^{\le 1}=0$ and hence $\nu=0$ in $\HH^{\odd}(\Gtl_\cA)$. Therefore the map is an embedding.

To construct $\epsilon$, let us first make sure that $ ν^1 (α) $ does not include the identity. Indeed let $ α: a → a $ be an indecomposable angle. Since $ a $ is not contractible by assumption, $ α $ is not the only angle in the polygon. In particular, $ α $ does not go from head to tail of $ a $ or from tail to head of $ a $. We conclude $ α $ is even and $ ν^1 (α) $ is odd, hence does not include an identity.

Regard the polygon that $ α $ sits in. By abuse of wording, let us say “head” and “tail” of a polygon's arc $ a $ to mean the vertex at the clockwise end and at the counterclockwise end of $ a $. Independent of arrow directions, $ ν^1 (α) $ can be decomposed into a part running from the tail of $ a $ to the tail of $ b $ and a part running from the head of $ a $ to the head of $ b $. Write this decomposition as $ ν^1 (α) = α δ^{(α)} + γ^{(α)} α $. Whether $ α $ is even or odd, $ δ^{(α)} $ and $ γ^{(α)} $ are always odd since $ ν^1 $ is.

Now let $ a $ be an arc. Denote by $ α_1 $, $ α_2 $, $ α_3 $, $ α_4 $ the angles incident at $ a $ as in \autoref{fig:hochschild-odd-arc-surrounding-angles}. We have
\begin{align*}
0 &= (dν) (α_2, α_4) = (-1)^{‖α_4‖} μ^2 (ν^1 (α_2), α_4) + μ^2 (α_2, ν^1 (α_4)) + ν^1 (μ^2 (α_2, α_4)) \\
&= - γ^{(α_2)} α_2 α_4 - α_2 δ^{(α_2)} α_4 + (-1)^{‖α_4‖} α_2 γ^{(α_4)} α_4 + (-1)^{‖α_4‖} α_2 α_4 δ^{(α_4)} + 0 \\
&= α_2 ((-1)^{‖α_4‖} γ^{(α_4)} - δ^{(α_2)}) α_4.
\end{align*}
We conclude $ δ^{(α_2)} = (-1)^{‖α_4‖} γ^{(α_4)} $. Note that we have used that $ α_2 α_4 = 0 $ and that $ γ^{(α_4)} $ and $ δ^{(α_2)} $ both run from the tail of $ a $ to the head of $ a $, and that $ α_4 $ ends at the tail of $ a $ and $ α_2 $ starts at the head of $ a $. Similarly,
\begin{align*}
0 &= (dν) (α_3, α_1) = (-1)^{‖α_1‖} μ^2 (ν^1 (α_3), α_1) + μ^2 (α_3, ν^1 (α_1)) + ν^1 (μ^2 (α_3, α_1)) \\
&= - α_3 δ^{(α_3)} α_1 - γ^{(α_3)} α_3 α_1 + (-1)^{‖α_1‖} α_3 α_1 δ^{(α_1)} + (-1)^{‖α_1‖} α_3 γ^{(α_1)} α_1 + 0 \\
&= α_3 (-δ^{(α_3)} + (-1)^{‖α_1‖} γ^{(α_1)}) α_1.
\end{align*}
We conclude $ δ^{(α_3)} = (-1)^{‖α_1‖} γ^{(α_1)} $. Let us now put
\begin{equation*}
ε^0_a ≔ (-1)^{‖α_4‖ + 1} γ^{(α_4)} - δ^{(α_3)} = (-1)^{‖α_1‖ + 1} γ^{(α_1)} - δ^{(α_2)}.
\end{equation*}
The expression can be read independent on the arrow direction of $ a $: It stays invariant under rotating the labels $ α_1 $, $ α_2 $, $ α_3 $ and $ α_4 $ by 180 degrees. In other words, $ ε^0_a $ can be seen either as the (signed) difference of $ γ $ and $ δ $ of its incident angles at its head or at its tail. This makes it easy to check $ (dε) (α) = ν^1 (α) $ for any indecomposable angle $ α: a → b $. Denote by $ β_1 $ and $ β_2 $ its predecessor and successor indecomposable angles around the same puncture as in \autoref{fig:hochschild-odd-angle-surrounding-angles}. Then
\begin{align*}
(dε) (α) &= (-1)^{|α|} ε^0_b α - α ε^0_a \\
&= (-1)^{|α|} ((-1)^{‖α‖ + 1} γ^{(α)} - δ^{(β_2)}) α - α ((-1)^{‖β_1‖ + 1} γ^{(β_1)} - δ^{(α)}) \\
&= γ^{(α)} α + (-1)^{‖α‖} δ^{(β_2)} α + (-1)^{‖β_1‖} α γ^{(β_1)} + α δ^{(α)} \\
&= γ^{(α)} α + α δ^{(α)} = ν^1 (α).
\end{align*}
In the penultimate equality, we have used that $ δ^{(β_2)} $ ends where $ α $ starts and is odd, therefore not composable with $ α $. Similarly $ α γ^{(β_1)} = 0 $. We conclude that $ (dε)^1 = ν^1 $ on indecomposable angles.
\end{proof}

\begin{figure}
\begin{center}
\begin{subfigure}{0.25\linewidth}
\centering
\begin{tikzpicture}
\path[draw, semithick, ->] (0, 0) -- ++(0, 1.5) node[midway, left] {$ a $} coordinate[pos=0.3] (alpha3-start) coordinate[pos=0.7] (alpha1-end);
\path[draw] (0, 1.5) -- ++(150:1) coordinate[pos=0.5] (alpha1-start);
\path[draw] (0, 1.5) -- ++(30:1) coordinate[pos=0.5] (alpha2-end);
\path[draw] (0, 0) -- ++(210:1) coordinate[pos=0.5] (alpha3-end);
\path[draw] (0, 0) -- ++(330:1) coordinate[pos=0.5] (alpha4-start);
\path[draw, bend right=65, ->] (alpha1-start) to (alpha1-end);
\path[draw, bend right=65, ->] (alpha1-end) to (alpha2-end);
\path[draw, bend right=65, ->] (alpha3-start) to (alpha3-end);
\path[draw, bend right=65, ->] (alpha4-start) to  (alpha3-start);
\draw (-.75,1.5) node{$\alpha_1$};
\draw (.75,1.5) node{$\alpha_2$};
\draw (-.75,0) node{$\alpha_3$};
\draw (.75,0) node{$\alpha_4$};
\end{tikzpicture}

\caption{Surrounding angles}
\label{fig:hochschild-odd-arc-surrounding-angles}
\end{subfigure}
\begin{subfigure}{0.25\linewidth}
\centering
\begin{tikzpicture}
\path[draw] (0, 0) -- ++(225:1.75) coordinate[pos=0.66] (beta2-start) -- ++(315:1.75) coordinate[pos=0.34] (beta1-end);
\path[draw, ->, bend right=30] (beta1-end) to node[midway, right] {$ α $} (beta2-start);
\path[draw, ->] (beta2-start) arc(45:120:0.6) node[midway, above] {$ β_2 $};
\path[draw, <-] (beta1-end) arc(315:240:0.6) node[midway, below] {$ β_1 $};
\end{tikzpicture}
\caption{Surrounding angles}
\label{fig:hochschild-odd-angle-surrounding-angles}
\end{subfigure}
\end{center}
\caption{}
\end{figure}

\begin{remark}
If $\nu$ is an even cocycle then $\nu^0$ is zero, but if $\nu$ has odd parity then $\nu^1$ need not to be zero or not even a derivation. An example of this occurs when $m$ is a marked point surrounded by a loop arc. This loop gives rise to an orbigon of length $1$, so the corresponding product $A_\infty$-deformation has a nontrivial $\rmu^1$ if $r_{(m,1)}\ne 0$ and therefore $\nu_{(m,1),o}$, as defined in definition \ref{defoddclass}, will be an odd cocycle with nontrivial $\nu^1$.
\end{remark}

\subsection{A set of generators}

In this section we will describe a special set of Hochschild classes that span the Hochschild cohomology.

The first class we need is the unit class $\nu_\id$, which corresponds to a single nullary product: $\nu_\id^0: \k \to A[1]: \id_a \mapsto \id_a$. Clearly this is a Hochschild cycle because the $\id_a$ are strict idempotents. Its parity is odd.

\begin{definition}\label{defoddclass}
For each orbifold point $p=(m,r)\in M \times \N$ we define its orbifold point homology class as
\[
\nu_{p,o} = \frac{1}{\hbar} (\rmu - μ) \text{ with }r = \ell_m^r \hbar \in Z(A)\otimes \C[\hbar]/(\hbar^2).
\]
\end{definition}

\begin{lemma}
$\nu_{p,o}$ is a cocycle of odd parity.
\end{lemma}
\begin{proof}
These are clearly cocycles because $\rmu$ satisfies the Maurer-Cartan equation, which reduces to $d(\rmu)=0$ in the first order.
\end{proof}

A second set of cocycles are characterized as follows. Fix a map $\lambda: (Q_\cA)_1 \to \C$. Define the first component of $\nu_\lambda$ to be the derivation such that $\nu_{\lambda}(\alpha)=\lambda_\alpha \alpha$ and set $\nu_\lambda^{\ne 1}=0$. The derivation condition implies that for
a path $\beta=\beta_k\dots\beta_1$ we have $\nu_\lambda(\beta) = \sum_i \lambda_{\beta_i}\beta$, so it makes sense to define $\lambda_\beta := \sum_i \lambda_{\beta_i}$. 

\begin{lemma}
$\nu_\lambda$ is a cocycle if and only if for every polygon $(\alpha_N,\dots,\alpha_1)$ we have $\sum_{i} \lambda_{\alpha_i}=0$.
\end{lemma}
\begin{proof}
The condition that $\sum_{i} \lambda_{\alpha_i}=0$ for every polygon implies that $\sum_{\alpha_i} \lambda_{\alpha_i}=0$ for every tree-gon $(\alpha_k,\dots,\alpha_1)$. Therefore
\[
(d\nu_\lambda)(\alpha_1,\dots, \alpha_k\beta)
=(\sum \lambda_{\alpha_i}+\lambda_\beta)\mu(\alpha_k,\dots,\alpha_1\beta)
-\nu_λ(\beta)= 0.
\]
On the other hand if $\nu^{\ne 1}=0$ then we know from lemma \ref{central-and-derivation} that $\nu$ must be a derivation, while 
for every polygon $d(\nu_\lambda)(\alpha_N,\dots, \alpha_1)=0$ implies that  $\sum_{i} \lambda_{\alpha_i}=0$.
\end{proof}

The set of all $ ν_\lambda$ with these properties form a vector space $ S $. If two elements in $S$ differ by a commutator with an element in $\k$ they will represent the same homology class. 

\begin{lemma}
$\dim  S / [\k, -] = \dim \H^1 (S, M, ℂ) = 2g - 1 + |M| $.
\end{lemma}
\begin{proof}
The dimension of $S$ is $2\#\cA - \#\{\text{faces}\}$ because there are $2$ angles arriving in each arc and every face gives one linear condition. All these conditions are linearly independent because every angle occurs only in one face. 
The kernel of the map  $\k \to [\k,-]$ is $\C$, so the image of $[\k,-]$ in $S$ is of dimension $\# \cA-1$ and therefore
\[
\dim  S / [\k, -] = 2\#\cA - \#\{\text{faces}\} - \#\cA +1= 2g-2 + |M| + 1.
\]
\end{proof}
It is easy to find a basis for $S/[\k, -]$. Look again at \autoref{fig:hochschild-odd-arc-surrounding-angles}.  For each arc $a$ there are 4 angles:
two $\alpha_2,\alpha_3$ leaving $a$ and two $\alpha_1,\alpha_2$ arriving in $a$. 
The angles $\alpha_2$ and $α_4$ sit in the face on the right of $a$ and $\alpha_1$,$α_3$ in the face on the left.

We define $\nu_{a,L}$ to be the derivation with 
\[
\lambda(\alpha_1)=1, \lambda(\alpha_3)=-1 
\]
while all other entries of $\lambda$ are zero. Similarly we define $\nu_{a,R}$ with
\[
\lambda(\alpha_2)=1, \lambda(\alpha_4)=-1.
\]
Note that because $[\id_a,-] = \nu_{a,L}-\nu_{a,R}$ we have that
up to homology $\nu_{a,L}=\nu_{a,R}$, so we will drop the subscript $R,L$ and
set $\nu_a:= \nu_{a,L}$.

\begin{lemma}
If $\cT$ is a set of arcs that forms a spanning tree in the face graph then
\[
\{ \nu_{a} \mid  a \in \cA\setminus \cT\}
\]
is a basis for $S / [\k, -]$.
\end{lemma}
\begin{proof}
The sum of the $\nu_{a}$ where $a$ runs over the arcs going around a given face is zero. Because $\cT$ is a spanning tree, the $\{ \nu_{a} \mid  a \in \cA\setminus T\}$ are independent.
Moreover because a spanning tree has $\#\{\text{faces}\}-1$ arcs, the cardinality of this set is the dimension of $S/[\k,-]$.
\end{proof}
\begin{remark}
We will call the $\nu_a$ and more general the $\nu_\lambda$, which are linear combinations of the $\nu_a$, \emph{arc classes}.
\end{remark}

Let $p =(m,j)$ be an orbifold point and $a$ be any arc arriving (or leaving) the marked point. We now define
\[
\nu_{p,e} := \nu_{p,o} \smile \nu_{a} \text{ or $ - \nu_{a}$ if $a$ leaves $m$.}
\]
Note that if $a$ is not a loop we have that $\nu_{p,e}$ is zero for every angle arrow except for $\alpha_1$ (or $\alpha_3$ in the case $a$ leaves $m$). 

\begin{lemma}
If the arc collection $\cA$ satisfies [NL2] then
the homology class of $\nu_{p,e}$ does not depend on the choice of $a$.
\end{lemma}
\begin{proof}
Let $ a $ and $ b $ be neighboring arcs incident at the same puncture, such that $ b $ comes after $ a $ in clockwise order. Let $ ε = ε^0 $ be the odd cochain given by $ ε^0 = ℓ_m^j \id_a $. For any angle $ α $ winding around puncture $ \mathrm{m} (α) $ we have
\[
d(ε)(\alpha) = \begin{cases}
+ \ell_m^j\alpha & t(\alpha)=a \text{ and } \mathrm{m} (α) = m \\
- \ell_m^j\alpha & h(\alpha)=a \text{ and } \mathrm{m} (α) = m \\
0& \text{otherwise}
\end{cases}
\]
Because $a$ is not a loop, the two nonzero cases each happen for just one $\alpha_1$ and $\alpha_3$ angle. We conclude
\[
[d(ε)]^1 = [\nu_{p,e} \smile \nu_{a} - \nu_{p,e} \smile \nu_{b}]^1.
\]
In other words the difference between $\nu_{p,e} \smile \nu_{a}$ and $\nu_{p,e} \smile \nu_{b}$ is homotopic to a cocycle $\kappa$ with $\kappa^0,\kappa^1=0$. 
In combination with theorem \ref{lowdegzeroiszero} this means that two consecutive arcs around $m$ define the same $\nu_{p,e}$-class, and hence all arcs around $m$ do.
\end{proof}

\begin{theorem}
\label{th:paper1-hochschild}
Let $\cA$ be an arc collection and $\cT$ a spanning tree in the face graph.
\begin{enumerate}
\item
The cocycles $\nu_\id$, $\nu_{p,o}$ where $p$ runs over all orbifold points form a basis for $\HH^{\odd}(\Gtl_\cA)$.
\item
If $\cA$ satisfies [NL2] then the cocycles $\nu_{p,e}$, $\nu_{a}$ where $p$ runs over all orbifold points and $a \in \cA \setminus \cT$ form a basis for $\HH^{\even}(\Gtl_\cA)$.
\end{enumerate}
\end{theorem}
\begin{proof}
The zeroth components of $\nu_\id$, $\nu_{p,o}$ form a basis for $Z(\Gtl_\cA)$, so the embedding $\HH^{\odd}(\Gtl_{\cA})\to Z(\Gtl_\cA)$ in lemma \ref{embeddings} is a bijection.

To show the second part, note that the fact that there are no loops or two-cycles implies that all nonzero angle paths with the same head and tail as $\alpha$ turn around the same marked point as $\alpha$ and are of the form $\ell_m^r\alpha$. 

If $\nu$ is $G$-homogeneous this means that there are two possibilities
\begin{itemize}
    \item If the degree is $0$ then $\nu$ must be a linear combination of the $\nu_{a}$. 
    \item $\nu(\alpha)=\lambda_\alpha\ell_m^j\alpha$ with $j>0$ then the only other $\beta$ for which $\nu(\beta)\ne 0$ must turn around $m$. Up to an inner derivation this $\nu$ is determined by the sum $\sum\lambda_\alpha$ where $\alpha$ runs over the indecomposables turning around $m$. 
\end{itemize}
\end{proof}

\subsection{Gerstenhaber structure product on Hochschild cohomology}
In this section, we compute the Gerstenhaber algebra structure on $ \HH(\Gtl_{\cA})$. In other words, we determine the bracket and the cup product on Hochschild cohomology. The bracket agrees with the computations by Wong for the Borel-Moore cohomology of matrix factorizations for dimer models \cite{wong2021dimer}, which is conjecturally equivalent to $\HH(\Gtl_\cA)$.

To describe the bracket and cup product, we use the basis for odd and even Hochschild cohomology constructed earlier: the odd classes $\nu_{(m,j),o}$ for a puncture $m \in M$ and $ j ≥ 0 $, the even classes $\nu_{(m,j),e}$, and the arc classes $ ν_\lambda$ with $\lambda: (Q_\cA)_1 \to \C$. 

We start with the cup-product
\begin{proposition}\label{calculatecup}
Let $ m,n \in M $ be two marked points, let $ i, j \ge 1 $ be two indices, and $ ν_\lambda, ν_\kappa$ two arc classes. Then the cup product in cohomology reads as follows:
\begin{alignat*}{2}
\nu_{(m,i),o} &\smile \nu_{(n,j),o} &&=  δ_{mn} \nu_{(m,i+j),o}, \\
\nu_{(m,i),o} &\smile \nu_{(n,j),e} &&=  δ_{mn} \nu_{(m,i+j),e}, \\
\nu_{(m,i),o} &\smile \nu_{\lambda} &&= \lambda_{\ell_m}\nu_{(m,i),e}, \\
\nu_{\kappa} &\smile \nu_{\lambda} &&= 0, \\
\nu_{(p,i),e} &\smile \nu_{\lambda} &&= 0. 
\end{alignat*}
\end{proposition}
\begin{proof}
We compute the cup products of the given Hochschild cocycles first on chain level. Then we compute their projection to cohomology. 
In fact, for the odd products it suffices to compute the curvature component $ (ν \smile η)^0 $ and for the even products it suffices to compute the first component $ (ν \smile η)^1 $. Also recall that the parity of $\smile$ itself is odd,
 so the product of two classes with the same (different) parity is odd (even). 
We are now ready to start the calculations. For the first identity, regard
\begin{equation*}
(\nu_{(m,i),o} \smile \nu_{(n,j),o})^0 = μ^2(\nu_{(m,i),o}^0(1),\nu_{(n,j),o}^0(1))=\delta_{mn} μ^2 (ℓ_m^i, ℓ_n^j) = ℓ_m^{i+j}.
\end{equation*}
This is precisely the curvature of the Hochschild cohomology class $\nu_{(m,i+j),o}$ and hence projects to it. 
For the second identity and third identity one can do a similar calculation but
now one has to look a the first component. Alternatively, we can use the definition of $\nu_{(p,i),e}$ in combination with the associativity of $\smile$ on the Hochschild cohomology. 

For the last two identities we again have to look at the zeroth component and these are trivially zero because both factors have trivial curvature. 
\end{proof}

\begin{proposition}\label{calculatebracket}
Let $ m,n \in M $ be two marked points, let $ i, j \ge 1 $ be two indices, and $ ν_\lambda, ν_\kappa$ two arc classes. Then the Gerstenhaber bracket in cohomology reads as follows:
\begin{align*}
[\nu_{(m,i),o}, \nu_{(n,j),o}] &= 0, \\
\left[\nu_{(m,i),e},\nu_{(n,j),o}\right] &= δ_{mn}· j · \nu_{(m,i+j),o},\\
\left[\nu_{(m,i),e},\nu_{(n,j),e}\right] &= δ_{mn} ·(j-i) · \nu_{(m,i+j),e}, \\
\left[ν_\lambda, \nu_{(m,i),o}\right] &= i λ_{ℓ_m} · \nu_{(m,i),o}, \\
\left[\nu_\lambda, \nu_{(m,i),e}\right] &= i λ_{ℓ_m} · \nu_{(m,i),e}, \\
[ν_\kappa, ν_\lambda] &= 0.
\end{align*}
\end{proposition}
\begin{proof}
The strategy for this calculation is the same as for the cup product. We calculate the zeroth or first component of the bracket. 
\begin{align*}
[ν, η]^0 &= ν^1 (η^0) - (-1)^{‖ν‖ ‖η‖} η^1 (ν^0), \\
[ν, η]^1 (α) &= ν^1 (η^1 (α)) + (-1)^{‖η‖ ‖α‖} ν^2 (η^0, α) + ν^2 (α, η^0) \\
&~ - (-1)^{‖ν‖ ‖η‖} \big(η^1 (ν^1 (α)) + (-1)^{‖ν‖ ‖α‖} η^2 (ν^0, α) + η^2 (α, ν^0)\big).
\end{align*}
The main difference is now that the bracket has even parity, so we need to check the zeroth component whenever the two entries have different parity. In that case
the result is the outer derivation of the even class applied to the central element of the odd class. This takes care of the second and fourth identity.

The first component of the bracket of two even classes is the commutator of their outer derivations because they have no $\nu^0$. This takes care of the third and the last two identities. 

Finally the bracket of two odd classes is zero because they have no $\nu^1,\nu^2$ if [NL2] holds. If [NL2] fails and $\alpha$ is an indecomposable angle then $\nu^1(\alpha)$ can only be nonzero if $a=h(\alpha)$ is a loop around a puncture but then the arc collection does not split the surface. Secondly for our deformed products a term of the form $\rmu^2 (\ell_p, α)$ is always cancelled by $\rmu^2 (α,\ell_p)$ and therefore the same holds for the odd classes. This implies that $[ν, η]^1 (α)$ is zero for all indecomposable $\alpha$ and because $[ν, η]^1$ is a derivation it is identically zero. 
\end{proof}

The Lie-bracket on $\HH^\bullet(\Gtl_\cA)$ is part of an $L_\infty$-structure.
This is a graded vector space $L$ together with brackets $[,\dots,]^k : L^{\otimes k} \to L$ of degree $2-k$ which satisfy the $L_\infty$-axioms \cite{stasheff2019infinity}.

If $(K^\bullet,d,[])$ is a DGLA then we can construct an $L_\infty$-structure $[,]^\bullet$ on the homology such that $(HK^\bullet,0,[,]^\bullet)$ is $ L_∞ $-quasi-isomorphic to $(K^\bullet,d,[])$.

These products can be constructed using the homotopy transfer lemma (see e.g. \cite{markl2004transferring}). The construction can be summarized as follows:
\begin{itemize}
    \item Split $K^\bullet$ as a direct sum of graded vector spaces
    \[
    K^\bullet = H\oplus I \oplus R
    \]
    such that $\mathrm{Im} d=I$ and $\Ker d=H\oplus I$. In this way $H\cong HK^\bullet$ and $d$ restricts to an isomorphism $d_{IR}: R \to I$.
    \item
     Let $h$ be the map 
     \[
     h : H\oplus I \oplus R \to H\oplus I \oplus R : (u,v,w) \mapsto (0,0,d^{-1}_{R→I}(v)).
     \]
     In this way the projection onto $H$ becomes $\pi= \id -dh-hd$.
     \item for $x_1,\dots,x_k \in H$ we define
     \[
     [x_1,\dots,x_k]^k = \sum_t c_t \pi[ \dots,h[x_i,x_{j}],\dots, ]
     \]
     as a linear combination of all possible ways of fully bracketed expressions where each internal bracket is composed with an $h$. (The precise coefficients $c_t$ will not matter in our discussion.)
\end{itemize}
In our case we will take $K=\HC(\Gtl_{\cA})$, $H\cong \HH(\Gtl_{\cA})$ to be the span of the classes we constructed, and we take $I$ and $R$ to be compatible with the $G$-grading.

\begin{theorem}
If $\cA$ satisfies [NL2] then we can choose a split such that all higher brackets are zero.
\end{theorem}
\begin{proof}
First note that $[\nu_{\id},-]=0$ in $\underline{\HC}(\Gtl_\cA)$, so $[\nu_{\id},-,\dots,-]^k$ will also be zero. 

To show that the other brackets are zero, pick any $\Z$-grading of the gentle algebra. Then we can assume that that the bracket $[,]^k$ has degree $2-k$ for this grading and degree $0$ for the $G=\H(S\setminus M,\cA,\Z)$-grading. As we assumed that $\cA$ has no loops, there are at least two punctures so the $G$-degree of $\ell_m$ is nontrivial for each marked point.

If all the entries of the bracket have $G$-degree $0$ (which means that they are all arc classes, which have $\Z$-degree $0$), then the homotopy transfer lemma
tells us that the product must be zero because all pairs $[\nu_\lambda,\nu_\kappa]$ are already zero in $\HC(\Gtl_\cA)$.

Suppose that there is at least one $\nu_{p,e}$ in one entry.
We make the following distinctions:
\begin{itemize}
    \item All the entries have $G$-degree a multiple of $\deg_G \ell_m$.
In that case the $G$-degree of the result will be $r\deg_G \ell_m$ for some $r \in \N$ and therefore the $\Z$-degree must be either $r\deg_{\Z}\ell_m$ if the result is even or $r\deg_{\Z}\ell_m-1$ if the result is odd. The total $\Z$-degree of the entries is at most $r\deg_{\Z}\ell_m$ and this only happens if all entries are all even. Because the product has degree $2-k$ it can only be nonzero for degree reasons if $2-k\ge-1$, or in other words if $k=2,3$.

\item Some entries have degrees that turn around different punctures.
If the result has $G$-degree $\deg_G \ell_m^j$ then the $G$-degree of the entries must also be $\deg_G \ell_m^j$ but if there are degrees for more than one marked point all marked points must appear because the only relation between the degrees of the $\ell_m$ is that $\sum_{m \in M}\deg_G\ell_m=0$. On the other hand by lemma \ref{totaldegreebelowzero} we have that $\sum_m\deg_\Z \ell_m\le -1$, so the $\Z$-degree of the result must be $\le \deg_\Z \ell_m^j +2 -k -1$. This is impossible if $k\ge 3$.
\end{itemize}

From the discussions above we see that because of degree reasons the only nontrivial higher brackets are of the form
\[
[\nu_{(m,i_1),e},\nu_{(m,i_2),e},\nu_{(m,i_3),e}]= \lambda\nu_{(m,i_1+i_2+i_3),o}
\]
or cases where one or more  $\nu_{(m,i_j),e}$ are $\nu_\lambda$s.
To show that these products are all zero, we have to look at the homotopy transfer construction in detail.
The triple product can be written down in terms of expressions of the form
\[
\pi[h[\nu_{(m,i_1),e},\nu_{(m,i_2),e}],\nu_{(m,i_3),e}]
\]
where $\pi$ is the projection onto homology.
We will now argue why such terms are zero.

From our computation in \ref{calculatebracket} we know that
\begin{equation*}\begin{split}
[\nu_{(m,i_1),e},\nu_{(m,i_2),e}]^0 &= 0\\
[\nu_{(m,i_1),e},\nu_{(m,i_2),e}]^1 &= ((i_2-i_1)\nu_{(m,i_1+i_1),e})^1
\end{split}\end{equation*}
So $[\nu_{(m,i_1),e},\nu_{(m,i_2),e}]= (i_2-i_1)\nu_{(m,i_1+i_2),e} + d\kappa$
with $(d\kappa)^{\le 1} = 0 $. From the construction in section \ref{reductiontozeroandfirstcomponent}, we can choose $\kappa$ such that $\kappa^0=0$.
For what follows, choose $ R^{\odd} = R_1 ⊕ R_2 $ (in a $ G $-graded way) such that $ R_1^{≠0} = R_2^0 = 0 $ and $ R_1 ∩ Z(\Gtl_{\cA}) = 0 $. With respect to the direct sum $ \HC^{\odd} = H^{\odd} ⊕ I^{\odd} ⊕ R_1 ⊕ R_2 $ write $ κ = h + i + r_1 + r_2 $. Then $ 0 = κ^0 = h^0 + r_1^0 $, hence $ r_1 = r_1^0 = - h^0 ∈ R_1 ∩ Z(\Gtl_{\cA}) = 0 $ and we conclude $ κ = i + r_2 $. Simply subtracting $ i $ from $ κ $ keeps $ κ^0 = 0 $ and brings $ κ $ into $ R_2 ⊂ R $. Finally
\[
[h[\nu_{(m,i_1),e},\nu_{(m,i_2),e}],\nu_{(m,i_3),e}]^0=[\kappa,\nu_{(m,i_3),e}]^0 = \nu_{(m,i_3),e}(\kappa^0) \pm \kappa^1(\nu_{(m,i_3),e}^0)=0
\]
Therefore all contributions are zero.
\end{proof}

\begin{corollary}\label{formality}
If $\cA$ has no loops or two-cycles then $\HC(\Gtl_\cA),d,[,]$ is formal. In other words there is an $L_\infty$-quasi-isomorphism between $\HC(\Gtl_\cA),d,[,]$ and $\HH(\Gtl_\cA),0,[,]$.
\end{corollary}

\subsection{Classifying curved deformations}
As an application of our computations we will show that every curved deformation of $\Gtl_\cA$ is equivalent to one of the $\rmu$. Remember that if $(A,\mu)$ is an $A_\infty$-algebra over $\k$ and $R$ is a complete local Noetherian unital commutative $\C$-algebra with maximal ideal $R^+$ and residue field $ R/R^+ = ℂ $ then a curved deformation is an odd element $\nu \in \Hom(\bigoplus_i A^{\otimes_\k^i},A)\htensor R^+$ such that $\mu + \nu$ satisfies the curved $A_\infty$-axioms. As we indicated before
this equation is equivalent to the Maurer-Cartan equation for the Hochschild cochain complex together with the Gerstenhaber bracket \cite{kontsevich2002deformation}.

\newcommand{\Art}{\operatorname{Art}}
\newcommand{\Set}{\operatorname{Set}}

\begin{theorem}
\label{th:paper1-classification}
If [NL2] holds, then every curved $A_\infty$-deformation of $ \Gtl_{\cA} $ over $ R $ is equivalent to one of the $ \rmu $.
\end{theorem}

\begin{proof}
We give a proof in terms of deformation functors, in the language of \cite{manetti2005deformation}. In short, we interpret our explicit class of deformations $ \rmu $ as a functor of Artin rings. Gauging by even elements acts on the values of this functor and lands in Maurer-Cartan elements. The first part of the proof deals with the case of $ R $ being Artinian. In the second part of the proof, we pass to the non-Artinian case.

Let $ \Art $ denote the category of Artinian local Noetherian unital commutative rings over $ ℂ $ with residue field $ ℂ $, with morphisms being local ($ φ(R^+) ⊂ S^+ $) and unital ($ φ(1_R) = 1_S $). We build three functors $ G, F, \MC: \Art → \Set $. The functor $ G $ is the gauge group functor, $ F $ is the functor of our deformation parameters $ r $, and $ \MC $ is the standard Maurer-Cartan functor. More precisely, define
\begin{align*}
G(R) &≔ \exp(\HC^{\even} (\Gtl_{\cA}) \htensor R^+), \\
F(R) &≔ Z(\Gtl_{\cA}) \htensor R^+, \\
\MC(R) &≔ \MC(\HC(\Gtl_{\cA}), R).
\end{align*}
All three assignments come with natural restriction maps $ G(R) → G(S) $, $ F(R) → F(S) $ and $ \MC(R) → \MC(S) $ for every morphism $ φ: R → S $ in $ \Art $. It is standard to check that all three are deformation functors in the sense of \cite[Definition 2.5]{manetti2005deformation}. In fact, $ G $ and $ F $ as well as their product functor $ G × F $ are smooth (unobstructed) in the sense of \cite[Definition 2.8]{manetti2005deformation}. Regard now the morphism of functors given by
\begin{equation*}
Φ(R): G(R) × F(R) → \MC(R), \\
(g, r) ↦ g.~\rmu.
\end{equation*}
Let us define obstruction theories $ O_{G × F} $ and $ O_{\MC} $ for $ G × F $ and $ \MC $ by defining both $ O_{G × F} ≔ 0 $ and $ O_{\MC} ≔ 0 $ to be trivial. We will now show that $ Φ $ is smooth by applying \cite[Proposition 2.17]{manetti2005deformation}. In the terminology of that paper, we have to check three items: (1) $ Φ(ℂ[ε]) $ is surjective where $ ℂ[ε] = ℂ[X] / (X^2) $ is the ring of dual numbers, (2) the obstruction theory $ O_{G × F} = 0 $ is complete, (3) the morphism between obstruction theories $ O_{G × F} → O_{\MC} $ is injective and compatible with $ Φ $.

We check the three conditions. (1) Let $ εc ∈ \MC(ℂ[ε]) $. Then $ c $ is a Hochschild cocycle and can be gauged by some $ g ∈ G(ℂ[ε]) $ to be equal to an $ \rmu $ for some $ r ∈ ε Z(\Gtl_{\cA}) $. In other words, we have $ Φ(ℂ[ε])(g, r) = εc $. This proves $ Φ(ℂ[ε]) $ surjective. Item (2) and item (3) are trivial because $ G × F $ is smooth.

The standard smoothness criterion \cite[Proposition 2.17]{manetti2005deformation} implies that $ Φ $ is smooth. Putting $ S = ℂ $ in the definition of smoothness implies that $ Φ(R): G(R) × F(R) → \MC(R) $ is surjective for every $ R ∈ \Art $. In other words, every deformation of $ \Gtl_{\cA} $ over $ R ∈ \Art $ is equivalent to an $ \rmu $ by gauge equivalence.

In the second part of the proof, we generalize to the non-Artinian case. Let $ R $ be a complete local Noetherian unital commutative $ ℂ $-algebra with residue field $ ℂ $. Let $ μ ∈ \MC(\HC(\Gtl_{\cA}), R) $ be a deformation over $ R $, meaning a Maurer-Cartan element in the completed tensor product $ μ ∈ \HC(\Gtl_{\cA}) \htensor R^+ $. Our aim is to show that $ μ $ is gauge equivalent to some $ \rmu $ with $ r ∈ Z(\Gtl_{\cA}) \htensor R^+ $.

Our strategy is to truncate $ μ $ to $ R/(R^+)^i $ for every $ i $ and use the first part of the proof to construct an element $ r_i $ and a gauge $ g_i $. We use smoothness of $ Φ $ to force both sequences $ (r_i) $ and $ (g_i) $ to converge.

Put $ μ_i ≔ π_i (μ) ∈ \MC(R / (R^+)^i) $. We shall construct sequences $ r_i ∈ F(R / (R^+)^i) $ and $ g_i ∈ G(R / (R^+)^i) $ such that (1) $ Φ(R / (R^+)^i)(g_i, r_i) = μ_i $, (2) $ π_i (r_{i+1}) = r_i $ and (3) $ π_i (g_{i+1}) = g_i $ for all $ i ∈ ℕ $. For the induction base $ i = 1 $, let $ g_1 ≔ 1 ∈ G(ℂ) $ and $ r_1 ≔ 0 ∈ F(ℂ) $. Since $ Φ(ℂ)(g_1, r_1) = 0 = μ_1 ∈ \MC(ℂ) $, the three conditions are satisfied at $ i = 1 $.

For induction hypothesis, assume the sequences have already been constructed until index $ i $. Since $ Φ $ is smooth, we have a surjection
\begin{equation*}
G(R / (R^+)^{i+1}) × F(R / (R^+)^{i+1}) \twoheadrightarrow (G(R / (R^+)^i) × F(R / (R^+)^i)) ×_{\MC(R / (R^+)^i)} \MC(R / (R^+)^{i+1}).
\end{equation*}
Pick $ (g_i, r_i, μ_{i+1}) $ on the right hand side. Indeed, $ Φ(R/(R^+)^i) (g_i, r_i) = μ_i = π_i (μ_{i+1}) $ by assumption and construction. By surjectivity there is a lift $ (g_{i+1}, r_{i+1}) $ such that (1) $ Φ(R / (R^+)^{i+1})(g_{i+1}, r_{i+1}) = μ^{i+1} $ and (2) $ π_i (r_{i+1}) = r_i $ and (3) $ π_i (g_{i+1}) = g_i $. This finishes the induction step.

Finally, we have constructed the desired sequences $ (r_i) $ and $ (g_i) $. Since $ π_i (r_{i+1}) = r_i $, the sequence $ r_i $ converges to some $ r ∈ Z(\Gtl_{\cA}) \htensor R^+ $ and $ g_i $ converges to some $ g ∈ \exp(\HC^{\even} (\Gtl_{\cA}) \htensor R^+) $. Within $ \MC(R/(R^+)^i) $ we have
\begin{equation*}
π_i (g.\rmu) = Φ(R/(R^+)^i)(g_i, r_i) = μ_i = π_i (μ), \quad ∀i ∈ ℕ.
\end{equation*}
Passing to the limit gives that $ g.\rmu = μ $ within $ \MC(R) $. In other words, $ μ $ is gauge equivalent to $ \rmu $.
\end{proof}

\begin{remark}
In remark \ref{Koszul1} we extended the notion of orbigons to allow weights on the faces and the marked points. This allows us to construct curved deformations ${}^{r,s}\bbmu$ of the (completed) gentle algebra without its $A_\infty$-structure coming from the Wrapped Fukaya category. It is also possible to show that every deformation of the completed gentle algebra is equivalent to one of these forms. This nicely fits into the framework of Koszul duality because Koszul dual $A_\infty$-algebras have the same deformation theory. 
\end{remark}

\printbibliography

@article{Bocklandt,
author={Bocklandt, Raf},
title={Noncommutative mirror symmetry for punctured surfaces},
note={With an appendix by Mohammed Abouzaid},
journal={Trans. Amer. Math. Soc.},
volume={368},
year={2016},
number={1},
pages={429--469}}

@article{HKK,
author={Haiden, F. and Katzarkov, L. and Kontsevich, M.},
title={Flat surfaces and stability structures},
journal={Publ. Math. Inst. Hautes \'{E}tudes Sci.},
volume={126},
year={2017},
pages={247--318}}

@article{Assem,
author={Assem, Ibrahim and Br\"{u}stle, Thomas and Charbonneau-Jodoin,
              Gabrielle and Plamondon, Pierre-Guy},
title={Gentle algebras arising from surface triangulations},
journal={Algebra Number Theory},
volume={4},
year={2010},
number={2},
pages={201--229}}

@article{kontsevich2002deformation,
  title={Deformation theory},
  author={Kontsevich, Maxim and Soibelman, Yan},
  journal={Livre en pr{\'e}paration},
  year={2002}
}

@article{manetti2005deformation,
  title={Deformation theory via differential graded Lie algebras},
  author={Manetti, Marco},
  journal={arXiv preprint math/0507284},
  year={2005}
}

@article{keller2006infinity,
  title={A-infinity algebras, modules and functor categories},
  author={Keller, Bernhard},
  journal={Contemporary Mathematics},
  volume={406},
  pages={67--94},
  year={2006},
  publisher={Providence, RI: American Mathematical Society}
}

@article{markl2004transferring,
  title={Transferring $A_\infty$ (strongly homotopy associative) structures},
  author={Markl, Martin},
  journal={arXiv preprint math/0401007},
  year={2004}
}

@article{stasheff2019infinity,
  title={L-infinity and A-infinity structures: then and now},
  author={Stasheff, Jim},
  journal={Higher Structures},
  volume={3},
  number={1},
  year={2019}
}

@article{seidel2008homological,
  title={Homological mirror symmetry for the genus two curve},
  author={Seidel, Paul},
  journal={arXiv preprint arXiv:0812.1171},
  year={2008}
}

@article{efimov2009homological,
  title={Homological mirror symmetry for curves of higher genus},
  author={Efimov, Alexander I},
  journal={arXiv preprint arXiv:0907.3903},
  year={2009}
}

@article{abouzaid2013homological,
  title={Homological mirror symmetry for punctured spheres},
  author={Abouzaid, Mohammed and Auroux, Denis and Efimov, Alexander and Katzarkov, Ludmil and Orlov, Dmitri},
  journal={Journal of the American Mathematical Society},
  volume={26},
  number={4},
  pages={1051--1083},
  year={2013}
}

@article{lekili2020derived,
  title={Derived equivalences of gentle algebras via Fukaya categories},
  author={Lekili, Yank{\i} and Polishchuk, Alexander},
  journal={Mathematische Annalen},
  volume={376},
  number={1},
  pages={187--225},
  year={2020},
  publisher={Springer}
}

@article{opper2018geometric,
  title={A geometric model for the derived category of gentle algebras},
  author={Opper, Sebastian and Plamondon, Pierre-Guy and Schroll, Sibylle},
  journal={arXiv preprint arXiv:1801.09659},
  year={2018}
}

@article{wong2021dimer,
  title={Dimer models and Hochschild cohomology},
  author={Wong, Michael},
  journal={Journal of Algebra},
  volume={585},
  pages={207--279},
  year={2021},
  publisher={Elsevier}
}

@article{bardzell1997alternating,
  title={The alternating syzygy behavior of monomial algebras},
  author={Bardzell, Michael J},
  journal={Journal of Algebra},
  volume={188},
  number={1},
  pages={69--89},
  year={1997},
  publisher={Elsevier}
}

@article{gerstenhaber1963cohomology,
  title={The cohomology structure of an associative ring},
  author={Gerstenhaber, Murray},
  journal={Annals of Mathematics},
  pages={267--288},
  year={1963},
  publisher={JSTOR}
}

@article{mescher2016primer,
  title={A primer on A-infinity-algebras and their Hochschild homology},
  author={Mescher, Stephan},
  journal={arXiv preprint arXiv:1601.03963},
  year={2016}
}

@book {Ganatra,
    AUTHOR = {Ganatra, Sheel},
     TITLE = {Symplectic {C}ohomology and {D}uality for the {W}rapped
              {F}ukaya {C}ategory},
      NOTE = {Thesis (Ph.D.)--Massachusetts Institute of Technology},
 PUBLISHER = {ProQuest LLC, Ann Arbor, MI},
      YEAR = {2012},
     PAGES = {(no paging)},
   MRCLASS = {Thesis},
}

@article{Ritter-Smith,
Author = {Alexander F. Ritter and Ivan Smith},
Title = {The monotone wrapped Fukaya category and the open-closed string map},
Year = {2012},
Eprint = {arXiv:1201.5880},
Doi = {10.1007/s00029-016-0255-9},
}

@article {BGS,
    AUTHOR = {Beilinson, Alexander and Ginzburg, Victor and Soergel,
              Wolfgang},
     TITLE = {Koszul duality patterns in representation theory},
   JOURNAL = {J. Amer. Math. Soc.},
  FJOURNAL = {Journal of the American Mathematical Society},
    VOLUME = {9},
      YEAR = {1996},
    NUMBER = {2},
     PAGES = {473--527},
      ISSN = {0894-0347},
}

\end{document}